\documentclass[smallextended]{svjour3}       
\usepackage{amsfonts}

\usepackage{latexsym}
\usepackage{amssymb}
\usepackage{amsmath}
\usepackage[mathscr]{eucal}
\usepackage{graphicx}
\usepackage{hyperref}
\usepackage{caption}
\usepackage{subcaption}
\usepackage[bottom=4cm, right=4cm, left=4cm, top=4cm]{geometry}

\renewcommand{\qed}{\hfill{\ \ \rule{2mm}{2mm}} \vspace{0.2in}}

\newcommand{\ind}{1\hspace{-2.3mm}{1}}

\renewcommand{\thefigure}{\arabic{figure}}
\begin{document}

\title{Constrained Maximum Weight Paths in Random Geometric Graphs}
\titlerunning{Maximum Weight Paths in RGGs}

\author{ \textbf{Ghurumuruhan Ganesan}}
\authorrunning{G. Ganesan}
\institute{IISER Bhopal\\
\email{gganesan82@gmail.com }}

\date{}
\maketitle

\begin{abstract}
In this paper, we consider a random geometric graph (RGG)~\(G\) formed by~\(n\) vertices distributed uniformly in the unit square~\(S\) on the plane and equip each edge of~\(G\) with an independent weight. We assume that the adjacency distance between vertices is larger than the connectivity threshold and estimate the maximum weight of a path connecting two fixed points~\(O_1\) and~\(O_2\) in~\(S,\) using edges of~\(G\) that  satisfy length and weight constraints. Our strategy is to first divide the space between~\(O_1\) and~\(O_2\) into small squares and identify nice squares containing heavy edges. We then use an iterative stitching procedure to connect these heavy edges and estimate the weight of the resulting path~\(P.\) Finally, we invoke a multi-level scaling procedure along with weight segmentation to establish  an upper bound for the maximum weight of \emph{any} path between~\(O_1\) and~\(O_2\) and thereby demonstrate the near-optimality of~\(P.\) We also illustrate our results using examples involving edge weights satisfying power law and exponential decay.

\vspace{0.1in} \noindent \textbf{Key words:} Maximum Weight Paths; Random Geometric Graphs; Independent Edge Weights; Edge Constraints;

\vspace{0.1in} \noindent \textbf{AMS 2000 Subject Classification:} Primary: 60C05;
\end{abstract}

\bigskip

\renewcommand{\theequation}{\thesection.\arabic{equation}}
\setcounter{equation}{0}
\section{Introduction} \label{intro}
Random Geometric Graphs (RGGs) are natural models for wireless networks~\cite{goldsmith} and have been studied in great detail from both theoretical~\cite{penrose}~\cite{penrose2} and application perspectives~\cite{gupta}~\cite{castro}. More recently, first passage percolation have also been studied in~\cite{hirsch}~\cite{lima} by considering edge weights that are proportional to the Euclidean length and shape theorems have been obtained for the ``progress" of such a spatial process.

Recently,~\cite{ganesan}~\cite{ganesan2}, has considered a variant of RGGs that are equipped with \emph{independent} edge weights. Indeed,~\cite{ganesan2} obtains bounds for the minimum weight of a spanning tree containing all the edges, when the adjacency distance is larger than the connectivity threshold and~\cite{ganesan}, estimates the maximum weight of a path between two fixed points in the unit square, when the edges have \emph{exponential} weights.

In this paper, we extend the study of maximum weight paths in RGGs for edge weights with general distributions. We consider  two fixed points~\(O_1\) and~\(O_2\) in the unit square and use the edges of an RGG~\(G\) for connecting~\(O_1\) and~\(O_2.\)  We  first use iteration techniques to construct an explicit path of large weight between~\(O_1\) and~\(O_2,\) consisting only of edges satisfying a minimum length and weight constraint.  We then invoke a multi-scaling argument to demonstrate the near optimality of our bound and illustrate  with examples involving edge weights satisfying exponential and power law decay.

The paper is organized as follows: In Section~\ref{sec_main}, we present our main results, Theorems~\ref{thm_path} and~\ref{thm_path2}, regarding the maximum weight of a path between two fixed points, formed by edges of an RGG with length and weight constraints. We also state Corollary~\ref{cor_one} where we illustrate the bounds obtained in the Theorems using examples of edge weights satisfying power law and exponential decay. Following that in Section~\ref{sec_prelim}, we collect preliminary results used in the proofs of our main Theorems. In Sections~\ref{sec_proof_a} and~\ref{sec_proof_b}, we prove Theorems~\ref{thm_path} and~\ref{thm_path2}, respectively and finally in Section~\ref{sec_proof_cor}, we prove Corollary~\ref{cor_one} using the estimates derived in Theorems~\ref{thm_path}-\ref{thm_path2}.

\renewcommand{\theequation}{\thesection.\arabic{equation}}
\setcounter{equation}{0}
\section{Main Results} \label{sec_main}
Let~\(K_n\) be the complete graph with vertex set~\(\{1,2,\ldots,n\}\) and let~\(\{X_v\}_{1 \leq i \leq n}\) be independent and identically distributed (i.i.d.) with a common density~\(f\) in the unit square~\(S = \left[-\frac{1}{2},\frac{1}{2}\right]^2\) satisfying
\begin{equation}\label{f_eq}
\epsilon_1 \leq f(x) \leq \epsilon_2
\end{equation}
for all~\(x \in S\) and some positive finite constants~\(\epsilon_1,\epsilon_2.\) We define~\(X_u\) to be the random \emph{location} of the vertex~\(u.\)

Let~\(K_{n}^{(2)}\) be the graph with vertex set~\(\{1,2,\ldots,n\} \cup \{O_1,O_2\}\) where~\(O_1\) and~\(O_2\) are deterministic vertices with locations~\(X_{O_1} = (0,0)\) (the origin) and~\(X_{O_2} = (\zeta_n,0) \in S,\) respectively, with~\(0 < \zeta_n  < \frac{1}{2}.\) The edge set of~\(K_n^{(2)}\) is obtained by connecting every vertex of~\(K_n\) with~\(O_1\) and with~\(O_2;\)  i.e., the edge set of~\(K_n^{(2)}\) is the union of the edge set of~\(K_n\) and~\(\{(O_1,j),(O_2,j)\}_{1 \leq j \leq n}.\) Here and henceforth we use the notation~\((u,v)\) to represent an edge with endvertices~\(u\) and~\(v.\)

For a deterministic sequence~\(0 < r_n < 1,\) let~\(G = G(r_n)\) be the random subgraph of~\(K_n^{(2)}\) formed by the set of all edges~\(h = (u,v)\) satisfying
\begin{equation}\label{rgg_def}
d(X_u,X_v) < r_n,
\end{equation}
where~\(d(a,b)\) represents the Euclidean distance between points~\(a\) and~\(b.\) We define~\(G\) to be the  Random Geometric Graph (RGG) with adjacency distance~\(r_n.\)

We equip the edges of~\(G\) with independent random weights as follows.  Let~\(\{W(h)\}_{h \in K_n^{(2)}}\) be positive i.i.d.\ random variables, indexed by the edge set of~\(K_n^{(2)},\) that are independent of all random variables defined so far. We refer to~\(W(h)\) as the \emph{weight} of edge~\(h\)  and define the complementary cumulative distribution function (ccdf)~\(F_c\) as
\begin{equation}\label{cdf_def}
F_c(x) := \mathbb{P}(W(h) > x)
\end{equation}
for~\(x > 0.\) For~\(0 < z < 1,\)  we define
\begin{equation}\label{h_def}
H(z) := \max\left\{x > 0 : F_c(x) \geq \frac{1}{z}\right\}
\end{equation}
to be the inverse cdf. Throughout, we assume that~\(0 < W(h) < \infty\) a.s.\ so that~\(0 < H(z) < \infty\) for all~\(0 < z < 1.\)


Our  goal is to connect the deterministic vertices~\(O_1\) and~\(O_2\) by paths in the RGG~\(G,\) having large weight. Formally, a path from~\(O_1\) to~\(O_2\) in~\(G\) is a sequence of vertices \[\Pi := (v_0 = O_1, v_1,v_2, \ldots, v_{t-1}, v_t = O_2)\] such that for each~\(0 \leq i \leq t-1,\) the edge~\((v_i,v_{i+1})\) is present in~\(G.\) We define the number of edges~\(t\) in~\(\Pi\) to be the \emph{length} of~\(\Pi\) and also denote~\(O_1\) and~\(O_2\) to be the endvertices of~\(\Pi.\)
\begin{definition}\label{def_one} For deterministic sequences~\(l_n,s_n,w_n \geq 0,\) we say that~\(\Pi\) is a~\((l_n,q_n,w_n)-\)path in~\(G\) if:\\
\((i)\)~\(\Pi\) contains at most~\(l_n\) edges and has~\(O_1\) and~\(O_2\) as endvertices,\\
\((ii)\) The Euclidean length of each edge in~\(\Pi\) is at least~\(q_n;\) i.e.,
\[\min_{0 \leq i \leq t-1} d(X_{v_i},X_{v_{i+1}}) \geq q_n\] and\\
\((iii)\) The weight of each edge in~\(\Pi\) is at least~\(w_n;\) i.e.,
\[\min_{0 \leq i \leq t-1} W(v_i,v_{i+1}) \geq w_n.\]
We define~\[W(\Pi) := \sum_{i=1}^{t-1}W(v_i,v_{i+1})\] to be the weight of~\(\Pi.\)
\end{definition}
In words, we seek the paths joining~\(O_1\) and~\(O_2,\) each of whose edges has length at least~\(s_n\) and weight at least~\(w_n.\)


Letting~\(M_n = M_n(l_n,s_n,w_n)\) denote the maximum weight of a~\((l_n,s_n,w_n)-\)path in~\(G\) (with the notation that the maximum of an empty set is zero), we have the following result. Throughout constants do not depend on~\(n\) and for two sequences~\(\{a_n\}\) and~\(\{b_n\},\) we use the notation~\(a_n = o(b_n)\) to denote that~\(\frac{a_n}{b_n} \rightarrow 0\) as~\(n \rightarrow \infty.\)
\begin{theorem}\label{thm_path} Let~\(\epsilon_1,\epsilon_2\) be as in~(\ref{f_eq}) and suppose
\begin{equation}\label{rn_cnd}
r_n \geq \frac{1}{n^{\beta}},\;\;\;\;q_n = \frac{1}{n^{c}}\;\;\;\text{ and }\;\;\; L_n := \frac{\zeta_n}{r_n} \geq 1,
\end{equation} for some constants~\(0 < \beta  < c < \frac{1}{2},\) strictly. Also suppose that the inverse cdf~\(H(.)\) is strictly increasing for all large~\(z.\) For every~\(\gamma,\kappa > 0\) and all~\(0< \varepsilon < 1\) small, there are constants~\(\theta_1,\theta_2  >0\) such that
\begin{equation}\label{weight_path_low}
\mathbb{P}\left(M_n(\theta_1 L_n,q_n,w_n) \geq \theta_2 L_n \alpha_n \right) \geq 1-\frac{1}{n^{1+\gamma}},
\end{equation}
where~\(w_n := H\left(N^{1-\varepsilon}\right),\;\;N := nr_n^2\) and
\[\alpha_n := \left\{\begin{array}{ll}
    H\left(\frac{N^2}{(\log{n})^{1+\kappa}}\right), & L_n \leq (\log{n})^{1+\kappa} \\
     & \\
    H(N^2),  &  \text{ otherwise.}
  \end{array}
  \right.\]
\end{theorem}
Since the length of any edge in~\(G\) is at most~\(r_n,\) we see that any path in~\(G\) connecting the vertices~\(O_1\) and~\(O_2\) that are~\(\zeta_n\) apart, must contain at least~\(L_n = \frac{\zeta_n}{r_n}\) edges. The above result essentially states that with high probability, i.e., with probability~\(1-o(1),\) there is a path between~\(O_1\) and~\(O_2\) having length at most of the order of~\(L_n\) and with weight least of the order of~\(L_n \alpha_n.\)  In this context, the term~\(\alpha_n\) could be interpreted as the gain obtained in choosing the maximum weight path as opposed to simply selecting a path according to a deterministic rule.

The following result complements Theorem~\ref{thm_path} by obtaining an upper bound for the maximum weight of a  path between the deterministic vertices~\(O_1\) and~\(O_2,\) having near optimal length, i.e., containing of the order of~\(L_n\) edges.
\begin{theorem}\label{thm_path2} Let~\(\epsilon_1,\epsilon_2\) be as in~(\ref{f_eq}) and suppose~\(r_n \geq \frac{1}{n^{\beta}}\) and~\(L_n = \frac{\zeta_n}{r_n} \geq 1\) for constant~\(0 < \beta < \frac{1}{2},\) as in~(\ref{rn_cnd}). Also suppose there are constants~\(C_0,x_0 > 0\) and~\(s > 4,\) such that the edge weight ccdf~\(F_c\) satisfies
\begin{equation}\label{f_scale}
F_c(ax) \leq \frac{C_0}{a^s} F_c(x)
\end{equation}
for all~\(a > 1\) and~\(x > x_0.\) For every~\(\kappa, \lambda > 0,\) there is a constant~\(\mu > 0\) such that
\begin{equation}\label{weight_path_up}
\mathbb{P}\left(M_n(\lambda L_n,0,0) \leq \mu L_n \alpha_n (1+\nu_n) \right) = 1-o(1),
\end{equation}
where~\(\alpha_n, N = nr_n^2\) are as in Theorem~\ref{thm_path2} and
\begin{equation}\label{nu_def_ax}
\nu_n :=
\left\{
\begin{array}{ll}
\left(L_n \cdot \log{n}\right)^{2/s}, & \;\;\;L_n  \leq (\log{n})^{1+\kappa},\\
 &\\
\frac{(\log{n})^{2+2\kappa+2/s}}{L_n^{1-2/s}}, &\;\;\;\text{ otherwise}. \\
\end{array}
\right.
\end{equation}
\end{theorem}
Thus we see that with high probability, any path containing order of~\(L_n\) edges, has weight at most of the order of~\(L_n\alpha_n(1+\nu_n).\)

We illustrate the above result with the following examples involving edge weights satisfying power law and exponential decay. 
\begin{corollary}\label{cor_one} Suppose that~\(r_n = \frac{1}{n^{\beta}}\) for some~\(0 < \beta < \frac{1}{2}.\) \\
\((a)\) Suppose the edge weight ccdf~\(F_c(x)\) is continuous for all large~\(x\) and satisfies the scaling relation~(\ref{f_scale}) for some~\(s > 4.\) Also suppose
\begin{equation}\label{nice_cond_ax}
q_n = r_n^{1+\varepsilon}\;\;\;\text{ and }\;\;\; L_n \geq (\log{n})^{b}
\end{equation}
for some constants~\(b > \frac{2(s+1)}{s-2}\) and~\(\varepsilon > 0.\) There are constants~\(\lambda_i> 0, i=1,2,3\) such that if~\(0 < \varepsilon < \lambda_1,\) then
\begin{equation}\label{dev_complete}
\mathbb{P}\left(\lambda_1L_n\alpha_n \leq M_n(\lambda_2 L_n,q_n,w_n) \leq  M_n(\lambda_2 L_n,0,0) \leq \lambda_3L_n\alpha_n\right)=1-o(1),
\end{equation}
where~\(L_n,\alpha_n\) and~\(w_n = w_n(\varepsilon)\) are as in Theorems~\ref{thm_path}-\ref{thm_path2}.\\
\((b)\) \emph{(Power Law Decay)} Suppose the edge weight ccdf~\(F_c\) is continuous and satisfies
\begin{equation}\label{f_power}
\frac{a_1}{x^s} \leq F_c(x) \leq \frac{a_2}{x^s}
\end{equation}
for some constants~\(a_1,a_2 > 0, s>4\) and all~\(x\) large. There are constants~\(\psi_i> 0, 1 \leq i \leq 4\) such that if~(\ref{nice_cond_ax}) holds for some~\(b > \frac{2(s+1)}{s-2}\) and~\(0 < \varepsilon < \psi_1,\) then
\begin{equation}\label{example_one}
\mathbb{P}\left(\psi_1 L_ng_n^2 \leq M_n\left(\psi_2 L_n,q_n,\psi_3g_n^{1-\varepsilon}\right) \leq M_n(\psi_2L_n,0,0) \leq \psi_4 L_ng_n^2\right) = 1-o(1),
\end{equation}
where~\(g_n := (nr_n^2)^{1/s}.\)\\
\((c)\) \emph{(Exponential Decay)} Suppose the edge weights are i.i.d.\ with a common ccdf~\(F_c(x) = \exp\left(-x^{\lambda}\right),\) for~\(x > 0\) and some constant~\(\lambda>0.\) There are constants~\(\kappa_i> 0, 1 \leq i\leq 4\) such that if~(\ref{nice_cond_ax}) holds for some~\(b > 2\) and~\(0 < \varepsilon < \kappa_1,\) then
\begin{equation}\label{example_two}
\mathbb{P}\left(\kappa_1 L_nh_n \leq M_n\left(\kappa_2 L_n,q_n,\kappa_3h_n\right) \leq M_n(\kappa_2L_n,0,0) \leq \kappa_4 L_nh_n\right) = 1-o(1),
\end{equation}
where~\(h_n := (\log{n})^{1/\lambda}.\)
\end{corollary}


\emph{Proof Outline}: We prove  the lower bound in Theorem~\ref{thm_path} by explicit construction in three steps. In the first step, we split the space between the fixed points~\(O_1\) and~\(O_2\) into small squares whose side length is of the order of~\(r_n\) and identify good squares containing ``heavy" edges. Following that, in the second step, we  show that the diameter of the subgraph of~\(G\) formed by ``nice" edges satisfying the edge and length constraint is constant in any small square,  with high probability. Finally, in the third step, we join the heavy edges identified in Step~\(1\) by nice paths constructed in Step~\(2,\) to obtain the desired lower bound.

For the upper bound in Theorem~\ref{thm_path2}, we use a multi-level scaling argument together with weight segmentation. We split the edges into different categories depending on the interval in which their weight lies and use appropriate scaling for each category to estimate the weight of edges in the maximum weight path. Finally, we derive the bounds in Corollary~\ref{cor_one} as a direct consequence of the bounds established in our main Theorems~\ref{thm_path}-~\ref{thm_path2}.


\renewcommand{\theequation}{\thesection.\arabic{equation}}
\setcounter{equation}{0}
\section{Preliminaries} \label{sec_prelim}
The following Lemma describes sufficient conditions under which ccdf~\(F_c(.)\) and the inverse ccdf~\(H(.)\) satisfy the inverse property and is used throughout the paper.
\begin{lemma}\label{lemma_hz} The following properties hold:\\
\((a)\) If there exists~\(z_0 > 0\) such that~\(H(z)\) is strictly increasing for all~\(z > z_0,\) then for any~\(z > z_0\) we have
\begin{equation}\label{inverse_cdf_equality}
\mathbb{P}\left(W(f) > H(z)\right) = \frac{1}{z}.
\end{equation}
\((b)\) If there exists~\(x_0 > 0\) such that~\(F_c(x_0) > 0\) and~\(F_c(x)\) is continuous for all~\(x \geq x_0,\) then~\(H(z)\) is strictly increasing for all~\(z > \frac{1}{F_c(x_0)}.\)
\end{lemma}

\emph{Proof of Lemma~\ref{lemma_hz}}: To prove~\((a),\) we let~\(z > z_0\) be arbitrary and  use the definition of~\(H(z)\) in~(\ref{h_def}) and the right continuity of the ccdf to get that~\[F_c(H(z)) = \mathbb{P}\left(W(f) > H(z)\right) \leq \frac{1}{z}.\] For the converse direction, we use the fact that~\(H(z)\) is strictly increasing for all~\(z > z_0,\) as mentioned in Lemma statement. This necessarily implies that  for any~\(z > z_0\) and  any~\(\varepsilon > 0,\)  we must have~\(F_c(H(z)) \geq \frac{1-\varepsilon}{z};\)  else we arrive at the contradictory relation~\(H(z) = H\left(\frac{z}{1-\varepsilon}\right).\) Combining the above, we get~(\ref{inverse_cdf_equality}) and this completes the proof of part~\((a)\) of the Lemma.

We prove part~\((b)\) by contradiction as follows.  Suppose there exists~\(u > z > z_0 := \frac{1}{F(x_0)}\) strictly such that~\(H(z) = H(u).\) Both~\(H(z)\) and~\(H(u)\) are necessarily at least~\(x_0\) and the right continuity of~\(F_c\) further implies that~\[ F_c(H(z)) = F_c(H(u)) \leq \frac{1}{u}.\] But since~\(\frac{1}{u} < \frac{1}{z}\) strictly, this implies that~\( F_c(H(z)) < \frac{1}{z}\) strictly and so  invoking the stronger continuity condition of the ccdf, we get that~\( F_c((1-\eta)H(z)) < \frac{1}{z}\) strictly, for all small~\(\eta> 0.\) This contradicts the definition of inverse ccdf in~(\ref{h_def}) and so~\(H(z)\) is strictly increasing for all~\(z > z_0.\) This completes the proof of the Lemma.~\(\qed\)

In our proof of Theorem~\ref{thm_path}\((b)\) below, we use an auxiliary result regarding the number of edges of given weight, ``located" in small squares contained within the unit square~\(S,\) described as follows. For integer~\(t\) satisfying
\begin{equation}\label{t_chic}
4 \leq t  \leq \frac{1}{100r_n},
\end{equation} split the unit square~\(S\) into disjoint~\(2tr_n \times 2tr_n\) squares~\(\{R_l\}_{1 \leq l \leq T_t}\)  as illustrated in Figure~\ref{fig_squares_ax} and assume for simplicity that~\(T_t := \frac{1}{4t^2r_n^2}\) is an integer; else we set the side length of~\(R_l\) to be~\(\frac{2tr_n}{A},\) where~\(A = A(n) \in [1,2)\) is chosen appropriately so that~\(\frac{A^2}{4t^2r_n^2}\) is an integer. This is possible since~\(t \leq \frac{1}{100r_n}\) and so
\begin{equation}\label{t_choice}
\frac{4}{4t^2r_n^2} - \frac{1}{4t^2r_n^2} = \frac{3}{4t^2r_n^2} \geq 1.
\end{equation}

\begin{figure}[tbp]
\centering
\includegraphics[width=3.5in, trim= 60 400 150 20, clip=true]{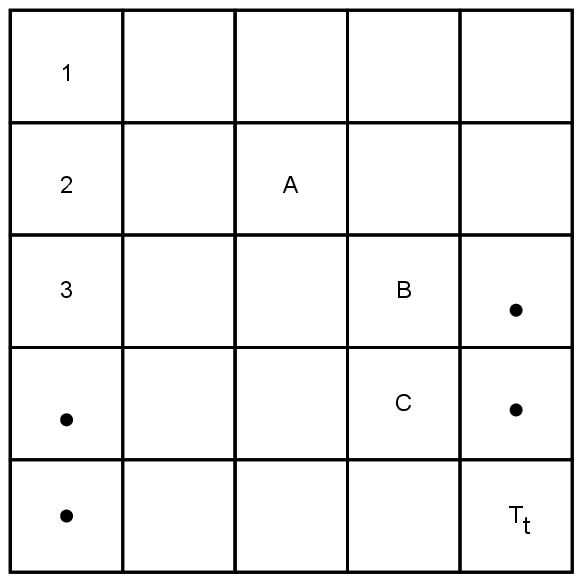}
\caption{Tiling the unit square into disjoint~\(2tr_n \times 2tr_n\) squares~\(\{R_l\}_{1 \leq l \leq T_t}.\) The square with label~\(l\) represents the~\(t-\)square~\(R_l\) and the~\(t-\)squares labelled~\(A,B\) and~\(C\) form a~\(t-\)animal of size~\(3.\)}
\label{fig_squares_ax}
\end{figure}

We define each~\(R_l\) to be a~\(t-\)\emph{square} and define a~\(t-\)\emph{animal} of size~\(Q\) to be a set~\[{\cal A} = \{B_i\}_{1 \leq i \leq Q} \subset \{R_l\}\] containing~\(Q\)~\(t-\)squares, whose union forms a connected set in the unit square~\(S.\) In Figure~\ref{fig_squares_ax},  the~\(t-\)squares labelled~\(A,B\) and~\(C\) form a~\(t-\)animal of size~\(3.\) We say that an edge~\((u,v)\) of~\(G\) is \emph{located} in the~\(t-\)square~\(B_l\) if both the endvertices~\(u\) and~\(v\) are located in~\(B_l;\) i.e.,~\(X_u,X_v \in B_l.\) We also say that~\({\cal A}\) contains the origin if some~\(t-\) square in~\({\cal A}\) contains the origin.

For~\(x > 0,\) let~\(N_{edge}(x,B_l)\) be the number of edges of~\(G\) located in~\(B_l,\) whose weight is at least~\(x\) and set
\[N_{edge}\left(x,{\cal A}\right) := \sum_{l=1}^{Q} N_{edge}(x,B_l)\] to be the total number of edges located in some~\(t-\)square of the~\(t-\)animal~\({\cal A},\) each with weight at least~\(x.\) Recalling that~\(N = nr_n^2\) (see statement of Theorem~\ref{thm_path}), we have the following Lemma.
\begin{lemma}\label{lem_animal} Suppose the conditions in Theorem~\ref{thm_path} hold and  let~\(F_c(.)\) be the edge weight ccdf. There are constants~\(\delta_1,\delta_2 > 0\) such that for any~\(x,y > 0\) and~\(Q \geq 1,\)
\begin{equation}
\mathbb{P}\left(\max_{\cal P} N_{edge}\left(x,{\cal P}\right) \geq y\right) \leq e^{\delta_1Q \Lambda} \cdot e^{-y} + e^{-\delta_2 N}, \label{tom_tom}
\end{equation}
where~\(\Lambda := \max(1,t^2N^2F_c(x))\) and the maximum is over all~\(t-\)animals of size~\(Q\) containing the origin.
\end{lemma}

In our proof of Lemma~\ref{lem_animal} below and throughout, we use the following deviation estimate regarding sums of independent Bernoulli random variables. Let~\(\{U_j\}_{1 \leq j \leq r}\) be independent Bernoulli random variables satisfying~\[\mathbb{P}(U_j = 1) = 1-\mathbb{P}(U_j = 0) > 0.\] If~\(V_r := \sum_{j=1}^{r} U_j, \theta_r := \mathbb{E}V_r\) and~\(0 < \gamma \leq \frac{1}{2},\) then
\begin{equation}\label{conc_est_f}
\mathbb{P}\left(\left|V_r - \theta_r\right| \geq \theta_r \gamma \right) \leq 2\exp\left(-\frac{\gamma^2}{4}\theta_r\right)
\end{equation}
for all \(r \geq 1.\) For a proof of~(\ref{conc_est_f}), we refer to Corollary A.1.14, pp. 312 of~\cite{alon}.

\emph{Proof of Lemma~\ref{lem_animal}}:  We begin with some preliminary computations regarding the number of vertices and edges located in a given~\(t-\)square. Formally, we say that vertex~\(u\) is located in~\(R_l\) if~\(X_u \in R_l.\) By the density bounds in~(\ref{f_eq}), we see that~\(X_u\) is located in the~\(2tr_n \times 2tr_n\) square~\(R_l\) with probability
\begin{equation}\label{gul_tip}
\int_{R_l} f \in [4\epsilon_1t^2r_n^2, 4\epsilon_2 t^2r_n^2],
\end{equation}
by~(\ref{f_eq}). The number of vertices~\(N(R_l)\) located in~\(R_l\) is Binomial with parameters~\(n\) and~\(\int_{R_l}f\) and so by the deviation estimate~(\ref{conc_est_f}), we then get that
\[\mathbb{P}\left(C_1t^2N \leq N(R_l) \leq C_2 t^2N\right) \geq 1-\exp\left(-2C_1 t^2N\right),\]
for some constants~\(C_1,C_2 > 0,\) where we recall from Lemma statement that~\(N = nr_n^2.\)

Defining
\[E_{sq} := \bigcap_{l=1}^{T_t} \left\{C_1 t^2N \leq N(R_l) \leq C_2 t^2N\right\}\] we get by an application of the union bound that
\begin{equation}\label{e_sq_ax}
\mathbb{P}(E_{sq}) \geq 1- T_t \exp\left(-2C_1t^2N\right).
\end{equation}
Using~\(t \geq 1,\) we get that the number of~\(t-\)squares~\(T_t\) satisfies \[T_t = \frac{1}{4t^2r_n^2} \leq \frac{1}{4r_n^2} \leq \frac{n^{2\beta}}{4},\] since~\(r_n \geq \frac{1}{n^{\beta}}\) by Theorem statement. Moreover,~\(t^2N \geq N = nr_n^2 \geq n^{1-2\beta} \rightarrow \infty\) and so plugging these into the final term of~(\ref{e_sq_ax}), we get that
\[T_t \exp\left(-2C_1t^2N\right) \leq \exp\left(-C_1t^2N\right) \leq \exp\left(-C_1N\right)\] for all~\(n\) large. Consequently,
\begin{equation}\label{e_vert_two_est}
\mathbb{P}(E_{sq}) \geq 1-e^{-C_1t^2N}.
\end{equation}

Next, to estimate the number of edges of~\(G\) located in~\(R_l,\) we first upper bound the degree of each vertex in~\(G.\) Let~\(B(X_u,r_n)\) be the ball of radius~\(r_n\) centred at the location~\(X_u\) of the vertex~\(u\)  so that the degree of vertex~\(u\) is precisely the number of vertices apart from~\(u,\) located in~\(B(X_u,r_n);\) i.e., the degree of~\(y\) in~\(G\) equals~\(N(B(X_u,r_n))-1.\) Given~\(X_u,\) we see that~\(N(B(X_u,r_n))\) is Binomial with parameters~\(n-1\) and~\[\int_{B(X_u,r_n)} f \in [\epsilon_1 \pi r_n^2, \epsilon_2 \pi r_n^2],\] again using the density bounds in~(\ref{f_eq}). Argue as in the discussion following~(\ref{gul_tip}) we then get that
\[\mathbb{P}\left(C_1 N \leq N(B(X_u,r_n)) \leq C_2 N \mid X_u\right) \geq 1-e^{-C_1N}\] for constants~\(C_1,C_2 > 0\) not depending on~\(X_u\) and so taking averages with respect to~\(X_u,\) we get
\[\mathbb{P}\left(C_1 N \leq N(B(X_u,r_n)) \leq C_2 N\right) \geq 1-e^{-C_1N}.\]

Defining
\[E_{ball} := \bigcap_{u=1}^{n} \left\{C_1 N \leq N(B(X_u,r_n)) \leq C_2 N\right\}\] and  arguing as in the derivation of~(\ref{e_vert_two_est})  above, we get that
\begin{equation}\label{e_ball_est}
\mathbb{P}(E_{ball}) \geq 1-e^{-C_1N}
\end{equation}
for all~\(n\) large. Finally, setting
\[A:= E_{sq} \bigcap E_{ball},\] we get from~(\ref{e_vert_two_est}),~(\ref{e_ball_est}) and the union bound that
\begin{align}
\mathbb{P}(A) &\geq 1- e^{-C_1t^2N} - e^{-C_1N} \nonumber\\
&\geq 1-2e^{-C_1N}, \label{p_a_est}
\end{align}
since~\(t \geq 2\) as set in~(\ref{t_chic}).

We assume henceforth that the event~\(A\) occurs and let~\(\mathbb{P}_A(.) := \mathbb{P}\left(. \mid A\right) \) be the distribution conditioned on the occurrence of the event~\(A.\) Because~\(E_{sq} \supset A\) occurs, the number of vertices located in the~\(t-\)square~\(R_l\) is at most~\(C_2t^2N\) and each such vertex is adjacent to at most~\(C_2N\) other vertices of~\(G,\) due to the occurrence of~\(E_{ball}.\) Consequently, the number of edges of~\(G\) located in~\(R_l\) is at most
\[C_2t^2N \cdot C_2 N = C_3 t^2N^2,\] where~\(C_3 = C_2^2 > 0\) is a constant. Each such edge has weight at least~\(x\) with probability~\(F_c(x),\) independent of other edges, by the definition of ccdf in~(\ref{cdf_def}). Thus~\(N_{edge}(x,R_l)\) is stochastically dominated from above by a Binomial random variable with parameters~\(C_3t^2N^2\) and~\(F_c(x)\) and so
\begin{align}
\mathbb{E}_A\left(\exp\left(N_{edge}(x,R_l)\right)\right) &\leq \left(1-F_c(x) + eF_c(x)\right)^{C_3t^2N^2} \nonumber\\
&\leq \exp\left(C_3(e-1)t^2N^2F_c(x)\right) \nonumber\\
&= \exp\left(C_4 t^2N^2F_c(x)\right), \label{dekha_tenu}
\end{align}
where~\(C_4 := C_3(e-1)\) is a constant.

If~\({\cal P} := \{B_i\}_{1 \leq i \leq Q}\) is any deterministic~\(t-\)animal of size~\(Q,\) then the corresponding random variables~\(\{N_{edge}(x,B_i)\}_{1 \leq i \leq Q}\) are~\(\mathbb{P}_A-\)independent and so
\begin{align}
\mathbb{E}_A\left(\exp\left(N_{edge}\left(x,{\cal P}\right)\right)\right) &= \prod_{i=1}^{Q} \mathbb{E}_A\exp\left(N_{edge}(x,B_i)\right) \nonumber\\
&\leq \exp\left(C_4t^2N^2QF_c(x)\right), \label{nasha_chadake}
\end{align}
by~(\ref{dekha_tenu}). The Chernoff bound implies that for~\(y > 0,\)
\begin{equation}\label{chenn}
\mathbb{P}_A\left(N_{edge}\left(x,{\cal P}\right) \geq y\right) \leq \exp\left(C_4t^2N^2QF_c(x)\right)e^{-y}.
\end{equation}


By standard estimates, the number of~\(t-\)animals of size~\(Q\) containing the origin is at most~\(e^{C_5Q}\) for some constant~\(C_5 > 0\) (see for e.g., the derivation of Eq.~\((4.24),\) pp.~\(76\) of~\cite{grimmett}).   Therefore invoking the union bound, we get from~(\ref{chenn}) that
\begin{align}\nonumber
\mathbb{P}_A\left(\max_{\cal P}N_{edge}\left(x,{\cal P}\right) \geq y\right) &\leq e^{C_5{Q}}\exp\left(C_4t^2N^2QF_c(x)\right) e^{-y} \nonumber\\
&\leq e^{C_6Q\Lambda}e^{-y}, \label{gilpin_tits}
\end{align}
for some constant~\(C_6 > 0,\) where~\(\Lambda = \max(1,t^2N^2F_c(x))\) is as in Lemma statement and  the maximum in~(\ref{gilpin_tits}) is over all~\(t-\)animals of size~\(Q\) containing the origin.

Finally, using
\begin{align}
\mathbb{P}(F) &= \mathbb{P}_A(F)\mathbb{P}(A) + \mathbb{P}(F \mid A^c) \mathbb{P}(A^c) \nonumber\\
&\leq \mathbb{P}_A(F) + \mathbb{P}(A^c) \label{pab_rel_two}
\end{align}
for any event~\(F\) and the estimate~(\ref{p_a_est}) for~\(\mathbb{P}(A^c),\) we get the desired bound~(\ref{tom_tom}) in the Lemma statement. This completes the proof of the Lemma.~\(\qed\)



\renewcommand{\theequation}{\thesection.\arabic{equation}}
\setcounter{equation}{0}
\section{Proof of Theorem~\ref{thm_path}} \label{sec_proof_a}
We begin with some preliminary constructions. Let~\(R_{path}\) be the rectangle with long side length~\(\zeta_n\) and short side length~\(\frac{r_n}{4},\) containing~\(A = (0,0)\) and~\(B = (\zeta_n,0)\) as corners. Tile~\(R_{path}\) into disjoint~\(\frac{r_n}{4} \times \frac{r_n}{4}\) squares~\(\{W_i\}_{1 \leq i \leq T}\) and assume for simplicity that~\(T = \frac{4\zeta_n}{r_n}\) is an integer; else we choose~\(4 < A = A(n) \leq 5\) in such a way~\(\frac{A\zeta_n}{r_n}\) is an integer and set~\(\frac{r_n}{A}\) to be the side length of each~\(W_i.\) This is possible since
\begin{equation}\label{a_choice}
\frac{5\zeta_n}{r_n} - \frac{4\zeta_n}{r_n} = \frac{\zeta_n}{r_n}  \geq 1,
\end{equation}
by Theorem statement. Also label the squares as shown in Figure~\ref{fig_squares} so that~\(W_i\) shares a common side with~\(W_{i+1}\) for~\(1 \leq i \leq T-1.\)

\begin{figure}[tbp]
\centering
\includegraphics[width=4.5in, trim= 20 500 50 110, clip=true]{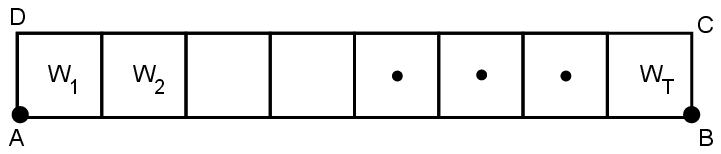}
\caption{Tiling the rectangle~\(R_{path} = ABCD\) with~\(A = (0,0)\) and~\(B = (\zeta_n,0),\) into~\(\frac{r_n}{4} \times \frac{r_n}{4}\) squares~\(\{W_i\}_{1 \leq i \leq T}.\)}
\label{fig_squares}
\end{figure}

For an \emph{irrational} constant~\(0 < \varepsilon < \frac{1}{4}\) to be determined later, say that an edge~\(h = (u,v)\) of~\(G\) is \emph{nice} if:\\
\((i)\) The length of~\(h\) is at least~\(q_n,\)\\
\((ii)\) The weight of~\(h\) is at least~\(w_n = H(N^{1-\varepsilon}),\) where~\(N := nr_n^2\) and\\
\((iii)\) There exists~\(1 \leq i \leq T\) such that both the endvertices of~\(h\) is present in~\(W_i;\) i.e., both~\(X_u\) and~\(X_v\) are in~\(W_i.\)

Similarly, say that an edge~\(h_{ab} = (a,b)\) of~\(G\) is a \emph{nice cross} edge if properties~\((i)-(ii)\) above hold and:\\
\((iv)\) There exists~\(1 \leq i \leq T-1\) such that one endvertex of~\(h_{ab}\) (say~\(a\)) is located in~\(W_i\) and the other endvertex~\(b\) is located in the adjacent square~\(W_{i+1}\) sharing a common side with~\(W_i.\)

Our strategy to establish the lower deviation bound in~(\ref{weight_path_low}) proceeds in three steps below.  In the first step, we define auxiliary events used in our proof and also show that a positive fraction of squares in~\(\{W_i\},\)  each contain a super heavy edge (formal definitions are provided in the analysis of Step~\(1\) below). Next, we demonstrate that with high probability, any two vertices in any~\(W_i\) can be connected together by a path of constant length, containing only nice edges. Finally,  we connect the super heavy edges using nice paths of constant length to obtain the desired deviation estimate in the Theorem statement. Details follow.



\emph{\underline{Step 1}}: For~\(1 \leq i \leq T,\) let~\({\cal V}(W_i)\) be the set of vertices located in the square~\(W_i\) and let~\(\Gamma_i\) be the induced subgraph of~\(G\) with vertex set~\({\cal V}(W_i),\) obtained by retaining only the nice edges. For~\(1 \leq i \leq T,\) let~\({\cal V}(W_i)\) be the set of vertices located in the square~\(W_i\) and let~\(\Gamma_i\) be the induced subgraph of~\(G\) with vertex set~\({\cal V}(W_i),\) obtained by retaining only the nice edges. In this step, we estimate the degree of each vertex in~\(\Gamma_i\) and define preliminary events used later in the construction of paths formed by nice edges.

The number~\(N(W_i) := \#{\cal V}(W_i)\)  of vertices of~\(\{X_j\}\) present in the square~\(W_i\) is Binomially distributed with parameters~\(n\) and
\begin{equation}\label{prob_bounds}
\int_{W_i} f \in \left[\frac{\epsilon_1 r_n^2}{16}, \frac{\epsilon_2 r_n^2}{16}\right],
\end{equation}
by the density bounds in~(\ref{f_eq}). Here and henceforth~\(\#A\) denotes the size of~\(A.\)  Consequently, the deviation estimate~(\ref{conc_est_f}) implies that
\begin{equation}\label{nwi_est}
\mathbb{P}\left(D_1 N \leq N(W_i) \leq D_2 N\right) \geq 1-\exp\left(-2D_1 N\right)
\end{equation}
for some constants~\(D_1,D_2> 0,\) where~\(N := nr_n^2.\)

Defining
\begin{equation}\label{e_vert_def}
E_{vert} := \bigcap_{i=1}^{T}\left\{D_1 N \leq N(W_i) \leq D_2 N\right\},
\end{equation}
we then get from~(\ref{nwi_est}) and the union bound that
\begin{equation}\label{e_vert_one}
\mathbb{P}(E_{vert}) \geq 1- T \cdot \exp\left(-2D_1 N\right).
\end{equation}
From Theorem statement
\begin{equation}\label{t_est}
N = nr_n^2 \geq  n^{1-2\beta} \rightarrow \infty\;\;\;\text{ and so }\;\;\;T = \frac{4\zeta_n}{r_n} \leq \frac{8}{r_n} \leq 2n
\end{equation}
for all~\(n\) large, where the penultimate inequality in~(\ref{t_est}) is true since the Euclidean distance between any two points in the unit square is at most~\(2.\) Plugging the bounds in~(\ref{t_est}) into~(\ref{e_vert_one}), we get that
\begin{align}\label{e_vert_est}
\mathbb{P}(E_{vert}) &\geq 1- 2n \cdot \exp\left(-2D_1 N\right) \nonumber\\
&\geq 1-\exp\left(-D_1 N\right) \nonumber\\
&= 1-o(1),
\end{align}
where the final two estimates in~(\ref{e_vert_est}) follow from the fact that~\(N = nr_n^2 \geq n^{1-2\beta} \rightarrow \infty,\) by Theorem statement.

For future reference, we remark here that if~\(E_{vert}\) occurs, then the following properties hold.\\
\(\textbf{(p1)}\) Any square~\(W_i\) contains at least~\(D_1N\) and at most~\(D_2N\) vertices.\\
\(\textbf{(p2)}\) Any two vertices in~\(W_i\) are adjacent in the graph~\(G.\)\\
\(\textbf{(p3)}\) Every vertex in~\(W_i\) is adjacent (in the graph~\(G\)) to at least~\(D_1N-1\) and at most~\(D_2N\) other vertices, in~\(W_i.\)\\
Indeed, property~\((p2)\) is true since the distance between any two points in~\(W_i\) is at most~\(r_n\) and the property~\((p3)\) follows from~\((p2).\)

To estimate the vertex degree in the smaller subgraph~\(\Gamma_i\) formed by nice edges, we use a couple of additional events. Say that an edge~\((u,v)\) of~\(G\) is \emph{long} if its length~\(d(X_u,X_v) > q_n = \varepsilon r_n\) and \emph{short} otherwise. To estimate the number of long edges adjacent to each vertex, let~\(1 \leq u \leq n\) be any vertex and define~\(B(X_u,q_n)\) to be the ball of radius~\(q_n\) centred at the location~\(X_u\) of vertex~\(u.\)

Irrespective of the location~\(X_u,\) at least one quadrant of~\(B(X_u,q_n)\) is completely contained within the unit square and so given~\(X_u,\) any vertex~\(X_j, j \neq u,\) is present in~\(B(X_u,q_n)\) with probability~\[\int_{B(X_u,q_n)}f \in \left[\frac{\epsilon_1 \pi q_n^2}{4}, \epsilon_2 \pi q_n^2\right],\] by the density bounds in~(\ref{f_eq}). Therefore,  the number~\(N(B(X_u,q_n))\) of vertices of~\(\{X_j\}\) present in~\(B(X_u,q_n)\) is stochastically dominated from above by a Binomial random variable with parameters~\(n\) and~\(\epsilon_2 \pi q_n^2\) and from below by a Binomial random variable with parameters~\(n\) and~\(\frac{\epsilon_1 \pi q_n^2}{4}.\) Using the deviation estimate~(\ref{conc_est_f}), we then get that
\begin{equation} \nonumber
\mathbb{P}\left(C_1nq_n^2 \leq N(B(X_u,q_n)) \leq C_2 nq_n^2 \,\middle \vert \, X_u\right) \geq 1-\exp\left(-C_1 nq_n^2\right),
\end{equation}
for some constants~\(C_1,C_2 > 0,\) not depending on~\(X_u,u\) or~\(\varepsilon.\) Taking averages with respect to~\(X_u,\) we obtain
\begin{equation}\nonumber
\mathbb{P}\left(C_1nq_n^2 \leq N(B(X_u,q_n)) \leq C_2n q_n^2\right) \geq 1-\exp\left(-C_1 nq_n^2\right)
\end{equation}
and further defining
\begin{equation}\label{e_short_def}
E_{short} := \bigcap_{u=1}^{n} \left\{C_1nq_n^2 \leq N(B(X_u,q_n)) \leq C_2 nq_n^2\right\},
\end{equation} we then get from the union bound that
\begin{align}\label{e_short_est}
\mathbb{P}(E_{short}) &\geq 1- n\cdot \exp\left(-C_1 nq_n^2\right) \nonumber\\
&= 1- n\cdot \exp\left(-C_1n^{1-2c}\right) \nonumber\\
&= 1-o(1),
\end{align}
since~\(q_n = \frac{1}{n^{c}}\) for some~\(0 < c < \frac{1}{2},\)  by Theorem statement.

From~(\ref{e_vert_est}),~(\ref{e_short_est}) and the union bound, we see that  the joint event~\(E_{vert} \cap E_{short}\) occurs  with probability
\begin{align}\label{e_vert_short_est}
\mathbb{P}\left(E_{vert} \cap E_{short} \right) &\geq 1- \exp\left(-D_1 N\right) - \exp\left(-C_1 n^{1-2c}\right) \nonumber\\
&\geq 1- 2\exp\left(-C_1 n^{1-2c}\right) \nonumber\\
&=1-o(1),
\end{align}
for all~\(n\) large, since~\[N = nr_n^2 \geq n^{1-2\beta} \geq n^{1-2c} \rightarrow \infty,\] by Theorem statement.

Suppose henceforth that~\(E_{vert} \cap E_{short}\) occurs and let~\(W_i, 1 \leq i \leq T\) be any square and let~\(u\) be any vertex located in~\(W_i;\) i.e.,~\(X_u \in W_i.\) From property~\((p2)\) above, we know that any two vertices in~\(W_i\) are adjacent in the graph~\(G\) and since~\(q_n = o(r_n)\) by Theorem statement, we see that~\(u\) is attached to at most
\begin{equation}\label{shortt_edges}
C_2nq_n^2 = o(N)
\end{equation} short edges in~\(G,\) i.e., edges of length at most~\(q_n.\) Moreover, from property~\((p3)\) above, we also have that~\(u\) is adjacent to at least~\(D_1N-1\) other vertices in~\(W_i\) in the graph~\(G.\)  Consequently, we deduce that~\(u\) is attached to at least~\[D_1N -1  - o(N)\] other vertices in~\(W_i\) by \emph{long} edges.

The above discussion implies that if~\(E_{vert} \cap E_{short}\) occurs, then there is a constant~\(D_0 > 0\) such that following properties hold for each square~\(W_i, 1 \leq i \leq T:\)\\
\(\textbf{(p4)}\) Any vertex~\(u\) located in~\(W_i\) is adjacent (in~\(G\)) to at least~\(D_0 N\) other vertices of~\(W_i,\) by long edges; i.e., edges of length at least~\(\varepsilon r_n.\)\\
\(\textbf{(p5)}\)  If~\(1 \leq i \leq T-1,\) then each vertex in~\(W_i\) is also adjacent to at least~\(D_0 N\) vertices by long edges, in the \emph{adjacent} square~\(W_{i+1}\) that shares a common side with~\(W_i,\) as shown in Figure~\ref{fig_squares}.

The next two events relate to the weight of edges present within the squares of~\(\{W_i\}.\) We begin with an  estimate on the number of  long edges attached to each vertex, that are also ``heavy". Recall that~\(q_n = \frac{1}{n^{c}}\) and~\(r_n \geq \frac{1}{n^{\beta}}\) for some constants~\(0 < \beta < c < \frac{1}{2}.\) Let~\(K\) be a large integer constant and~\(\frac{1}{K} < \varepsilon < \frac{1}{K-1}\) be an irrational constant satisfying
\begin{equation}\label{k_chce}
1-2c < (1-2\beta)\left(1-\frac{1}{K}\right) < (1-2\beta) \varepsilon (K-1).
\end{equation}
Say that an edge~\(f\) with endvertices~\(u\) and~\(v\) is \emph{heavy} if its weight~\[W(f) > w_n := H(N^{1-\varepsilon}).\]  From the relation~(\ref{inverse_cdf_equality}) of Lemma~\ref{lemma_hz}, we know that  any  edge of~\(G\) is heavy with probability
\begin{equation}\label{prob_est}
p := \mathbb{P}\left(W(f) > w_n\right) =  \frac{1}{N^{1-\varepsilon}},
\end{equation}
independently of other edges. By definition, any edge of~\(G\) that is long, heavy and has both endvertices in~\(W_i,\) is nice  and we  recall that~\(\Gamma_i \subset G\) is the subgraph obtained by retaining all nice edges of~\(G\) present in~\(W_i.\) We also recall that~\({\cal V}(W_i),\) the set of all vertices in~\(W_i,\) is the vertex set of~\(\Gamma_i.\)

Suppose as before that~\(E_{vert} \cap E_{short}\) occurs so that the square~\(W_i\) contains at least~\(D_1N\) and at most~\(D_2N\) vertices, by property~\((p1)\) above.  If~\({\cal S} \subset \{1,2,\ldots,n\}\) is a deterministic set of size~\(\#{\cal S} \in [D_1N,D_2N],\)  then  given the occurrence of the joint event~\(E_{vert} \cap E_{short}\) and the event~\({\cal V}(W_i) = {\cal S},\) the degree~\(d_i(v)\) of a vertex~\(v \in {\cal S}\) in~\(\Gamma_i\) is stochastically dominated from below by a Binomial random variable with parameters~\(D_0N\) and~\(p,\) by property~\((p4)\) above. Also since~\({\cal S}\) has at most~\(D_2N\) vertices, we see that~\(d_i(v)\) is stochastically dominated from above by a Binomial random variable with parameters~\(D_2N\) and~\(p.\)



Let~\({\cal F}_X\)  be the sigma-field generated by the vertex locations~\(\{X_j\},\) so  that the events~\(E_{vert},E_{short}\) and the random set~\({\cal V}(W_i)\) are~\({\cal F}_X-\)measurable. Defining
\[E_{heavy}(i,v) := \left\{\frac{D_0Np}{2} \leq d_i(v) \leq 2D_2Np\right\}\] and using the deviation estimate~(\ref{conc_est_f}), we get from the discussion in the previous paragraph that
\begin{align}\label{e_nice_loc_est}
&\mathbb{P}\left(E_{heavy}^c(i,v) \,\middle \vert\, {\cal F}_X\right) \ind({\cal V}(W_i) = {\cal S}) \ind\left(E_{vert} \cap E_{short}\right) \nonumber\\
&\;\;\;\;\;\leq\;\;\;\exp\left(-D_3 Np\right) \ind({\cal V}(W_i) = {\cal S}) \nonumber\\
&\;\;\;\;\;=\;\;\;\exp\left(-D_3 N^{\varepsilon}\right) \ind({\cal V}(W_i) = {\cal S}),
\end{align}
for some constant~\(D_3 > 0\) not depending on the choice of~\({\cal S},\) where the final estimate in~(\ref{e_nice_loc_est}) is true since~\(p = \frac{1}{N^{1-\varepsilon}},\) by~(\ref{prob_est}).

Further defining
\[E_{heavy}(i) := \bigcap_{v \in {\cal V}(W_i)}\left\{\frac{D_0Np}{2} \leq d_i(v) \leq 2D_2Np\right\}, \]
we get from~(\ref{e_nice_loc_est}) and the union bound that
\begin{align}\label{e_nice_loc_est_two}
&\mathbb{P}\left(E_{heavy}^c(i) \,\middle \vert\, {\cal F}_X\right) \ind({\cal V}(W_i) = {\cal S}) \ind(E_{vert} \cap E_{short}) \nonumber\\
&\;\;\;\;\;\leq\;\;\;\#{\cal S} \cdot \exp\left(-D_3 N^{\varepsilon}\right) \ind({\cal V}(W_i) = {\cal S}) \nonumber\\
&\;\;\;\;\;\leq\;\;\;n \cdot \exp\left(-D_3N^{\varepsilon}\right) \ind({\cal V}(W_i) = {\cal S}),
\end{align}
where the final estimate in~(\ref{e_nice_loc_est_two}) is true since the size of any vertex set is at most~\(n.\)

Taking averages in~(\ref{e_nice_loc_est_two}), we get that
\begin{equation}
\mathbb{P}\left(E_{heavy}^c(i)\bigcap \left\{{\cal V}(W_i) = {\cal S}\right\}\bigcap E_{vert} \bigcap E_{short}\right) \leq n \cdot e^{-D_3N^{\varepsilon}} \mathbb{P}\left({\cal V}(W_i) = {\cal S}\right) \nonumber
\end{equation}
and summing over all possible values of~\({\cal S},\) we get that
\begin{equation}\label{e_nice_loc_est_4}
\mathbb{P}\left(E_{heavy}^c(i)\bigcap E_{vert} \bigcap E_{short}\right) \leq n \cdot \exp\left(-D_3N^{\varepsilon}\right).
\end{equation}

Using
\begin{align}
\mathbb{P}(A^c) &= \mathbb{P}(A^c \cap B) + \mathbb{P}(A^c \cap B^c) \nonumber\\
&\leq \mathbb{P}(A^c \cap B) + \mathbb{P}(B^c) \label{pab_rel}
\end{align}
with~\(A = E_{heavy}(i)\) and~\(B = E_{vert} \cap E_{short}\) and invoking the estimate~(\ref{e_vert_short_est}) for the event~\(B,\) we get that
\begin{equation}\label{jalsa}
\mathbb{P}(E_{heavy}^c(i)) \leq n \cdot \exp\left(-D_3N^{\varepsilon}\right) + \exp\left(-C_1 n^{1-2c}\right).
\end{equation}
Defining
\[E_{heavy} := \bigcap_{i=1}^{T} E_{heavy}(i),\] we then get by an application of the union bound that
\begin{align}\label{e_heavy_est}
\mathbb{P}(E^c_{heavy}) &\leq  n \cdot T \exp\left(-D_3N^{\varepsilon}\right) +  T \cdot \exp\left(-C_1n^{1-2c}\right) \nonumber\\
&\leq \exp\left(-D_4n^{b}\right),
\end{align}
for some constants~\(D_4,b > 0,\) since the number of squares~\(T \leq 4 n\) (see~(\ref{t_est})) and~\(N = nr_n^2 = n^{1-2\beta} \geq n^{1-2c} \rightarrow \infty,\) by Theorem statement.

As before, for future reference, we remark that if~\(E_{vert} \cap E_{short} \cap E_{heavy}\) occurs, then the following property holds for each square~\(W_i,1 \leq i \leq T:\)\\
\(\textbf{(p6)}\) The degree~\(d_i(v)\) of a vertex~\(v\) in~\(\Gamma_i\) satisfies
\begin{equation}\label{d_low_up_est}
2d_{low} := \frac{D_0Np}{2} = \frac{D_0N^{\varepsilon}}{2} \leq d_i(v) \leq 2D_2Np = 2D_2N^{\varepsilon} =: d_{up}.
\end{equation}
\(\textbf{(p7)}\) Each edge in~\(\Gamma_i\) is nice.\\

We use properties~\((p6)-(p7)\) later to extract heavy paths containing few edges, from each~\(W_i.\) We also recall (see the beginning of this proof) that our strategy is to connect these heavy paths  using nice cross edges; i.e., long and heavy edges having one endvertex in~\(W_i\) and the other endvertex in the adjacent  square~\(W_{i+1}.\)  We therefore define~\(E_{cross}\) to be the event that for each~\(1 \leq i \leq T-1,\) there is a nice cross edge having one endvertex in~\(W_i\) and other endvertex in~\(W_{i+1}.\) Arguing as in the derivation of~(\ref{e_heavy_est}) and using the property~\((p5)\) that each vertex in~\(W_i\) is also adjacent to at least~\(D_0 N\) vertices by long edges in~\(W_{i+1},\) we get that
\begin{equation}\label{e_cross_est}
\mathbb{P}(E^c_{cross})\leq \exp\left(-D_5 N^{\varepsilon}\right) = o(1)
\end{equation}
for some constant~\(D_5 > 0.\)

If~\(E_{vert} \cap E_{short} \cap E_{heavy} \cap E_{cross}\) occurs, then the following property holds:\\
\(\textbf{(p8)}\) For every~\(1 \leq i \leq T-1,\) there is a nice cross edge having one endvertex in~\(W_i\) and other endvertex in~\(W_{i+1}.\)\\



The final ingredient concerns the maximum weight of a long edge present within any square~\(W_i.\) If~\(E_{vert} \cap E_{short}\) occurs, then from property~\((p4)\) above we know  that there is a constant~\(D_0 > 0\) such that any vertex~\(v\) in~\(W_i\) is adjacent to at least~\(D_0N\) other vertices in~\(W_i,\) by long edges, i.e., edges of length at least~\(\varepsilon r_n.\) Also, by property~\((p1)\) there are at most~\(D_2N\) vertices in~\(W_i\) for some constant~\(D_2 > 0\) and so~\(v\) is adjacent to at most~\(D_2N\) other vertices in~\(W_i.\) The standard handshaking argument implies that the total number of edges in any graph is twice the sum of its vertex degrees and so, the total number of long edges having both endvertices in~\(W_i\) is at least~\(\gamma_1 N^2\) and at most~\(\gamma_2 N^2,\) for some constants~\(\gamma_1,\gamma_2 > 0.\)

Suppose first that~\(L_n = \frac{\zeta_n}{r_n} \geq (\log{n})^{1+\kappa}\) so that~\(\alpha_n\) defined in Theorem statement equals~\(H(N^2),\) where~\(H(.)\) is the inverse ccdf as defined in~(\ref{h_def}). If~\(\Lambda_i\) denotes the maximum weight of a long edge with both endvertices in~\(W_i,\) then for any~\(x > 0\)
\begin{align}
\mathbb{P}\left(\Lambda_i \leq x  \mid E_{vert} \cap E_{short}\right) &\leq \left(\mathbb{P}\left(W(h) \leq x\right)\right)^{\gamma_1 N^2}   \nonumber\\
&= \left(1-\mathbb{P}(W(h) > x)\right)^{\gamma_1 N^2}.\label{tnjb}
\end{align}
Setting~\(x = \alpha_n = H(N^2),\) we get from the  relation~(\ref{inverse_cdf_equality}) that
\begin{equation}\label{jilla}
\mathbb{P}\left(W(h) > \alpha_n\right) = \frac{1}{N^2}
\end{equation} and so
\begin{equation} \label{tulsa}
\mathbb{P}\left(\Lambda_i \leq \alpha_n  \mid E_{vert} \cap E_{short}\right) \leq  \left(1-\frac{1}{N^2}\right)^{\gamma_1 N^2} \leq e^{-\gamma_1}.
\end{equation}

We recall that there are~\(T = 4L_n = \frac{4\zeta_n}{r_n}\) squares in~\(\{W_i\}_{1 \leq i \leq T}\) and given the joint event~\(E_{vert} \cap E_{short},\)  the events~\(\{\Lambda_i \geq H(N^2)\}_{1 \leq i \leq T}\) are independent and occur with probability at least~\(1-e^{-\gamma_1} > 0\) each, by~(\ref{tulsa}). For~\(\theta > 0\) we therefore define the event~\[E_{lambda}(\theta) := \left\{\sum_{i=1}^{T} \ind(\Lambda_i \geq H(N^2)) \geq \theta T \right\}\] and use the deviation estimate~(\ref{conc_est_f}), to get that
\begin{equation}\label{lambda_estv}
\mathbb{P}\left(E_{lambda}(\theta_1) \mid E_{vert} \cap E_{short}\right) \geq 1-e^{-\theta_2 L_n},
\end{equation}
for some constants~\(\theta_1,\theta_2 > 0.\) Setting~\(E_{lambda} := E_{lambda}(\theta_1)\) and using the estimate~(\ref{e_vert_short_est}) for the event~\(E_{vert} \cap E_{short},\) we argue  as before (see for e.g., the derivation of~(\ref{jalsa})) to get that
\begin{align}\label{tasker}
\mathbb{P}\left(E_{lambda} \right) &\geq 1-e^{-\theta_2 L_n} - e^{-2\theta_3N }\nonumber\\
&\geq 1-e^{-\theta_3 L_n} - e^{-\theta_3 N}
\end{align}
for some constant~\(\theta_3 > 0.\)

Say that an edge~\(h = (u,v)\) of~\(G\) is \emph{super nice} if~\(h\) is long, its weight~\(W(h) > \alpha_n\) and there exists~\(1 \leq i \leq T\) such that both~\(u\) and~\(v\) are located in~\(W_i.\)  We know that~\(N = nr_n^2 \geq n^{1-2\beta} \rightarrow \infty\) by Theorem statement and also the inverse cdf~\(H(z)\) is strictly increasing for all~\(z\) large, again by Theorem statement. Therefore any super nice edge is also nice, i.e., satisfies the properties~\((i)-(iii)\) stated in the paragraph following~(\ref{a_choice}). Moreover,~\(W_i\) contains a super nice edge if and only if~\(\Lambda_i > H(N^2)\) and so summarizing, we see that if the event~\(E_{lambda}\) occurs, then the following property holds:\\
\(\textbf{(p9)}\) There are at least~\(\theta_1 T\) squares in~\(\{W_i\}_{1 \leq i \leq T},\) each of which  contain a super nice edge.\\
\(\textbf{(p10)}\) Any super nice edge is also nice.\\

Finally, defining~\[E_{tot} :=  E_{vert} \bigcap E_{short} \bigcap E_{heavy}  \bigcap E_{cross} \bigcap E_{lambda},\] we get from~(\ref{e_vert_short_est}),~(\ref{e_heavy_est}),~(\ref{e_cross_est}),~(\ref{tasker}) and the union bound that
\begin{align}\label{e_tot_est2}
&\mathbb{P}(E_{tot}) \nonumber\\
&\;\;\geq\;\;1- \exp\left(-C_1n^{1-2c}\right) - \exp\left(-D_4n^{b}\right) - \exp\left(-D_5 N^{\varepsilon}\right)- \exp\left(-\theta_3 N\right)-e^{-\theta_3 L_n} \nonumber\\
&\;\;\geq\;\;1-\exp\left(-\theta_4 n^{z}\right) -e^{-\theta_4 L_n},
\end{align}
for some constants~\(\theta_4,z > 0,\) again using the fact that~\(N = nr_n^2 \geq n^{1-2\beta}  \geq n^{1-2c} \rightarrow \infty,\) by Theorem statement. Since~\(L_n \geq (\log{n})^{1+\kappa},\)  we then get from~(\ref{e_tot_est2}) that
\begin{equation}\label{e_tot_est}
\mathbb{P}(E_{tot}) \geq 1-\frac{1}{n^{1+\gamma}}
\end{equation}
for any constant~\(\gamma > 0\) and all~\(n\) large.

If~\(L_n \leq (\log{n})^{1+\kappa},\) we argue as above with minor modifications, after the estimate~(\ref{tnjb}). Indeed, in this case the term~\(\alpha_n\) in Theorem statement is~\(\alpha_n = H\left(\frac{N^2}{(\log{n})^{1+\kappa}}\right)\) and so instead of~(\ref{jilla}), we have
\begin{equation}\label{jilla_22}
\mathbb{P}\left(W(h) > \alpha_n\right) = \frac{(\log{n})^{1+\kappa}}{N^2}
\end{equation} and  instead of~(\ref{tulsa}) we have
\begin{equation} \label{tulsa_22}
\mathbb{P}\left(\Lambda_i \leq \alpha_n  \mid E_{vert} \cap E_{short}\right) \leq  \left(1-\frac{(\log{n})^{1+\kappa}}{N^2}\right)^{\gamma_1 N^2} \leq e^{-\gamma_1(\log{n})^{1+\kappa}}.
\end{equation}
There are~\(T \leq 2n\) squares in~\(\{W_i\}_{1 \leq i \leq T}\) (see~(\ref{t_est})) and  by a direct union bound, we get in place of~(\ref{lambda_estv}) that
\begin{align}\label{lambda_estv_22}
\mathbb{P}\left(E_{lambda}(\theta_1) \mid E_{vert} \cap E_{short}\right) &\geq 1-T \cdot e^{-\gamma_1(\log{n})^{1+\kappa}}\nonumber\\
&\geq 1- 2n\cdot e^{-\gamma_1(\log{n})^{1+\kappa}}
\end{align}
Arguing as in the analysis following~(\ref{lambda_estv}), we again get that~(\ref{e_tot_est}) holds.

The definition of the event~\(E_{tot}\) above and its corresponding probability estimate~(\ref{e_tot_est}), completes the first step of our proof.

\emph{\underline{Step 2}}: Recall that~\(\Gamma_i \subset G, 1 \leq i \leq T\) is the induced graph of~\(G\) whose vertex set is the set of all vertices~\({\cal V}(W_i)\)  located in the~\(\frac{r_n}{4} \times \frac{r_n}{4}\) square~\(W_i\) and the edge set is the set of all nice edges of~\(G;\) i.e., edges that are both long and heavy (see the beginning of the proof of Step~\(1\)). In this step, obtain an upper bound for the diameter of a slightly modified version of~\(\Gamma_i.\)

Formally, for an integer constant~\(\Delta \geq 1,\)  let~\({\cal Y} \subset \{1,2,\ldots, n\}\) be any set of~\(\Delta\) deterministic vertices and let~\(\Gamma_i\left({\cal Y}\right)\) be the induced subgraph obtained by removing the vertices of~\({\cal Y}\) from~\(\Gamma_i.\) The graph distance between two vertices~\(a\) and~\(b\) in~\(\Gamma_i\left({\cal Y}\right)\) is the minimum number of edges in a path in~\(\Gamma_i\left({\cal Y}\right),\) containing~\(a\) and~\(b\) as endvertices. Here we use the notation that the minimum of an empty set is~\(\infty.\) The diameter~\(diam\left(\Gamma_i\left({\cal Y}\right) \right) \) is the maximum graph distance between two vertices in~\(\Gamma_i\left({\cal Y}\right).\)

For constant~\(\mu > 0,\) let
\[F_{diam}(\mu, \Delta) := \bigcap_{i=1}^{T} \bigcap_{{\cal Y}} \left\{diam\left(\Gamma_i\left({\cal Y}\right)\right) \leq \mu \right\} \] be the event that the diameter of~\(\Gamma_i\left({\cal Y}\right)\) is at most~\(\mu\) for any~\(i\) or~\({\cal Y}.\) Recalling that each edge of~\(\Gamma_i\) has length at least~\(\varepsilon r_n\) (see Step~\((1)\) above), our main estimate in this step is the following: For any irrational~\(0 < \varepsilon < \frac{1}{4}, \Delta \geq 1\)  and~\(\gamma > 0,\) we have that
\begin{equation}\label{diam_est_am}
\mathbb{P}\left(F_{diam}(K,\Delta)\right) \geq 1 - \frac{1}{n^{1+\gamma}},
\end{equation}
for all~\(n\) large, where~\(K = K(\varepsilon)\) is the smallest integer strictly larger than~\(\varepsilon^{-1},\) as defined in~(\ref{k_chce}).

Our strategy to prove~(\ref{diam_est_am}) is as follows. Let~\( 1 \leq u \leq n\) be \emph{any} vertex different from the vertices in~\({\cal Y}.\) Set~\({\cal N}_0(u) = \{u\}\) and for integer~\(l \geq 1,\) let~\[{\cal N}_l(u) := {\cal N}_l(u,{\cal Y}) \subset {\cal V}(W_i) \setminus {\cal Y}\] be the set of all vertices at a graph distance of~\(l\) from~\(u\) in~\(\Gamma_i\left({\cal Y}\right).\) Also let~\(\Gamma^{(L)}_i \subset G\) be the induced subgraph of~\(G\) with vertex set~\({\cal S},\) containing only long edges; i.e., edges with length at least~\(s_n.\) By definition, we see that~\[\Gamma_i({\cal Y}) \subset \Gamma_i \subset \Gamma^{(L)}_i \subset G.\]

Let~\({\cal S} \subset \{1,2,\ldots,n\}\) be any deterministic subset of size~\(\#{\cal S} \in [D_1N, D_2N]\) and let~\(\Gamma\) be any deterministic graph with vertex set~\({\cal S}.\) We assume  that
\begin{equation}\label{fs_event_def}
F({\cal S},\Gamma) := \left\{{\cal V}(W_i) = {\cal S}\right\} \bigcap \left\{\Gamma^{(L)}_i = \Gamma\right\}
\end{equation}
occurs and deduce~(\ref{diam_est_am}) by showing that with high probability, for any~\(u \in {\cal S}\setminus {\cal Y},\) every vertex in~\(\Gamma_i({\cal Y})\) belongs to~\({\cal N}_l(u)\) for some~\(1 \leq l \leq K.\)



We begin with some preliminary definitions and computations. For~\(l \geq 1,\) let
\begin{equation}\label{rl_def}
{\cal R}_l(u) := {\cal V}(W_i) \setminus  \left({\cal Y} \bigcup \bigcup_{j=0}^{l-1} {\cal N}_j(u)\right)
\end{equation}
be the set of all vertices that are at a graph distance of at least~\(l\) from~\(u\) in~\(\Gamma_i\left({\cal Y}\right).\) For an integer constant~\(M \geq 1\) to be determined later, let~\(Q_l(M,u)\) be the event that every vertex~\(v \in {\cal N}_l(u)\) satisfies the following three properties:\\
\((i)\)~\(v\) is adjacent to at most~\(M\) vertices in~\({\cal N}_l(u),\)\\
\((ii)\)~\(v\) is adjacent to at  most~\(M\) vertices in~\({\cal N}_{l-1}(u)\) and\\
\((iii)\)~\(v\) is adjacent to at least~\(d_{low}\) and at  most~\(d_{up}\) vertices in~\({\cal R}_{l+1}(u),\) where~\(d_{low}\) and~\( d_{up}\) defined in~(\ref{d_low_up_est}) respectively relate to the lower and upper bounds for the vertex degree in~\(\Gamma_i,\) as derived in property~\((p6)\) of Step~\(1\) above.\\
Set
\begin{equation}\label{h_l_def_ax}
H_{l}(u) := \bigcap_{j=0}^{l}\left\{\frac{d^{j}_{low}}{M^{j-1}} \leq \#{\cal N}_{j}(u) \leq d_{up}^{j}\right\}
\end{equation}
and let~\(E_{vert}\) and~\(E_{short}\) be the events defined in Step~\(1.\) We finally define the events~\[I_0(u) = \Lambda_0(u) := E_{vert} \bigcap E_{short}\] and
\begin{equation}\label{i_def}
\Lambda_l(u) := E_{vert} \bigcap E_{short} \bigcap H_l(u)\;\;\text{ and }\;\;I_l(u) := \Lambda_l(u) \bigcap  Q_{l}(M,u),
\end{equation}
for~\(l \geq 1.\)

We assume henceforth that~\(E_{vert} \cap E_{short}\)   occurs and get from property~\((p1)\) that~\({\cal V}(W_i)\) has size at least~\(D_1N\) and  at most~\(D_2N,\) for some constants~\(D_1,D_2 > 0,\) where~\[D_1N =D_1nr_n^2 \geq D_1n^{1-2\beta} \rightarrow \infty\] by Theorem statement. We therefore let~\({\cal S} \subset \{1,2,\ldots,n\}\) be any deterministic subset of size~\(\#{\cal S} \in [D_1N, D_2N]\) and let~\(\Gamma\) be any graph with vertex set~\({\cal S}\) that is consistent with the occurrence of~\(E_{vert} \cap E_{short}\) and assume that the event~\(F({\cal S},\Gamma)\) defined in~(\ref{fs_event_def}) holds.

For irrational~\(0 < \varepsilon < \frac{1}{4},\) let~\(K(\varepsilon)\) be the largest integer satisfying~\(\frac{1}{K} > \varepsilon,\) as defined in~(\ref{k_chce}). As a first step towards establishing~(\ref{diam_est_am}), we show below that for every~\(\gamma > 0,\) the following estimates hold. For every~\({\cal Y} \subset \{1,2,\ldots,n\}\) containing~\(\Delta\) vertices and every~\(u \in {\cal S} \setminus {\cal Y},\)  we have that
\begin{equation}\label{h_gen_est}
\mathbb{P}\left(I^c_l(u) \bigcap I_{l-1}(u) \mid F({\cal S},\Gamma)\right) \leq \frac{3}{n^{\gamma+\Delta+3}},
\end{equation}
for each~\(1 \leq l \leq K-1.\)

The derivation of~(\ref{h_gen_est}) in turn is facilitated by  auxiliary estimates stated in the next paragraph and proved in the subsequent paragraphs. Specifically, for constant integer~\(M \geq 1\) to be determined later, we set~\(E_l(M,u)\) to be the event that every vertex in~\({\cal R}_l(u)\) is adjacent to at most~\(M\) vertices in~\({\cal N}_l(u).\) Similarly let~\(F_l(M,u)\) be the event that every vertex in~\({\cal R}_l(u)\) is adjacent to at most~\(M\) vertices in~\({\cal N}_{l-1}(u)\) and finally define~\(\Psi_l(M,u)\) to be the event that every vertex in~\({\cal N}_{l}(u)\) is adjacent to at least~\(2d_{low}\) and at most~\(d_{up}\) vertices, in~\({\cal R}_{l}(u).\)

Our strategy to prove~(\ref{h_gen_est}) is as follows. We first argue that if
\begin{equation}\label{e_joint_def}
E_{joint}(u) :=  E_l(M,u) \bigcap \Psi_l(M,u) \bigcap F_l(M,u)   \bigcap I_{l-1}(u) \bigcap F({\cal S},\Gamma)
\end{equation} occurs, then necessarily so does~\(I_l(u).\) We then show that
\begin{equation}
\mathbb{P}\left(F_l^c(M,u) \bigcap I_{l-1}(u) \mid   F({\cal S},\Gamma)\right)  \leq \frac{1}{n^{\Delta+\gamma+3}}, \label{fl_est_app}
\end{equation}
\begin{equation}
\mathbb{P}\left(E_l^c(M,u) \bigcap F_l(M,u) \bigcap I_{l-1}(u) \mid F({\cal S},\Gamma)\right) \leq \frac{1}{n^{\Delta+\gamma+3}}, \label{el_est_app}
\end{equation}
and
\begin{equation}
\mathbb{P}\left(\Psi_l^c(M,u) \bigcap F_l(M,u) \bigcap I_{l-1}(u) \mid F({\cal S},\Gamma)\right) \leq \frac{1}{n^{\Delta+\gamma+3}} \label{psi_l_est_app}
\end{equation}
for all~\(n\) large. Applying the union bound, we then get~(\ref{h_gen_est}).

Indeed, the occurrence of the events~\(E_l(M,u)\) and~\(F_l(M,u)\) respectively, imply the properties~\((i)\) and~\((ii)\)  in the definition of~\(Q_l(M,u)\) above. The event~\(\Psi_l(M,u)\) ensures that  any vertex~\(v \in {\cal N}_{l}(u)\) has at least~\(2d_{low}\) and at most~\(d_{up}\) neighbours in~\({\cal R}_{l},\) where
\begin{equation}\label{zimaad}
2d_{low} = \frac{D_0Np}{2} = \frac{D_0N^{\varepsilon}}{2} \rightarrow \infty
\end{equation}
since~\(N = nr_n^2 \geq n^{1-2\beta} \rightarrow \infty,\) by Theorem statement. Because~\(E_{l}(M,u)\) also occurs, we know that~\(v\) is adjacent to at most~\(M\) other vertices in~\({\cal N}_l(u)\) and is  therefore adjacent to at least~\[2d_{low}- M\geq d_{low}\] vertices in~\({\cal R}_{l+1},\) for all~\(n\) large. This implies that property~\((iii)\) in the definition of~\(Q_l(M,u)\) also holds.

Moreover, because~\(I_{l-1}(u)\) occurs, we know that~\({\cal N}_{l-1}(u)\) has size at most~\(d_{up}^{l-1}\) and property~\((iii)\) in the definition of~\(Q_{l-1}(M,u) \supset I_{l-1}(u)\) implies that each vertex in~\({\cal N}_{l-1}(u)\) is adjacent to at most~\(d_{up}\) other vertices in~\({\cal R}_l(u).\) Therefore~\({\cal N}_l(u)\) necessarily has size
\begin{equation}\label{zimax5}
\#{\cal N}_l(u) \leq  d_{up} \#{\cal N}_{l-1}(u) \leq d^{l}_{up}.
\end{equation}
To demonstrate that~\(I_l(u)\) occurs, it is therefore enough to lower bound the size of~\({\cal N}_l(u).\) For~\(l=1,\) the discussion following~(\ref{zimaad}) implies that
\begin{equation}\label{n_one_est}
d_{up} \geq \#{\cal N}_1(u) \geq d_{low} = \frac{D_0N^{\varepsilon}}{4}
\end{equation}
and so we conclude that~\(I_1(u)\) occurs.

To estimate the size of~\({\cal N}_l(u)\) for~\(l \geq 2,\) we proceed by iteration. Since~\(I_{l-1}(u)\) occurs, we know that
\begin{equation}\label{ind_est}
\frac{d^{j}_{low}}{M^{j-1}} \leq \#{\cal N}_{j}(u) \leq d_{up}^{j}
\end{equation}
for each~\(1 \leq j \leq l-1.\) Let~\(K_{l-1,l}(u)\) be the bipartite subgraph of~\(\Gamma_i\left({\cal Y}\right)\) with left vertex set~\({\cal N}_{l-1}(u)\) and right vertex set~\({\cal N}_l(u).\) For any vertex~\(v \in {\cal N}_{l-1}(u),\) let~\(d_{left}(v)\) be the degree of~\(v\) in~\(K_{l-1,l}(u)\) and similarly for a vertex~\(w \in {\cal N}_{l-1}(u),\) let~\(d_{right}(w)\) be the degree of~\(w\) in~\(K_{l-1,l}(u).\) By the standard handshaking argument, we have that the total number of edges in~\(K_{l-1,l}(u)\) is the sum of the left vertex degrees and is also equal to the sum of right vertex degrees and so
\begin{equation}\label{hand_shake}
\sum_{v \in {\cal N}_{l-1}(u)} d_{left}(v) = \sum_{w \in {\cal N}_{l}(u)}d_{right}(w).
\end{equation}

Because~\(Q_{l-1}(M,u) \supset I_{l-1}(u)\) occurs, each vertex~\(v \in {\cal N}_{l-1}(u)\) contains at most~\(M\) neighbours in~\({\cal N}_{l-1}(u)\) and at most~\(M\) neighbours in~\({\cal N}_{l-2}(u)\) (see properties~\((i)-(ii)\) in the definition of~\(Q_{l-1}(M,u) \) prior to~(\ref{i_def})). Similarly, property~\((iii)\) ensures that~\(v\) has  at least~\(d_{low}\) neighbours in~\({\cal R}_{l}(u).\)  Thus
\begin{equation}\label{d_left_est}
d_{left}(v) \geq d_{low}.
\end{equation}
and so
\begin{equation}\label{wamiqa_one}
\sum_{v \in {\cal N}_{l-1}(u)} d_{left}(v) \geq d_{low} \cdot \#{\cal N}_{l-1}(u).
\end{equation}
Similarly, the occurrence of~\(F_{l}(M,u)\)  ensures that~\(d_{right}(w) \leq M\) for each right vertex~\(w \in {\cal N}_l(u)\) and so
\begin{equation}\label{wamiqa_two}
\sum_{w \in {\cal N}_l(u)} d_{right}(w) \leq M \cdot \#{\cal N}_l(u).
\end{equation}
Combining~(\ref{wamiqa_two}) with~(\ref{wamiqa_one}) and using~(\ref{hand_shake}), we then get that
\begin{equation}\label{n_two_est}
\#{\cal N}_l(u) \geq \frac{d_{low}}{M} \cdot \#{\cal N}_{l-1}(u) \geq \frac{d_{low}^l}{M^{l-1}},
\end{equation}
by~(\ref{ind_est}). This obtains the desired lower bound for the size of~\({\cal N}_l(u)\) and we get that if~\(E_{joint}(u)\) occurs, then so does~\(I_l(u).\)

In what follows, we estimate the probability of each of the three events~\(E_l(M,u), F_l(M,u)\) and~\(\Psi_l(M,u)\) separately. Recalling that~\({\cal F}_X\) is the sigma-field generated by the vertex locations~\(\{X_j\},\) we let~\({\cal G}_l = {\cal G}_l\left(u,{\cal Y}\right)\) be the sigma-field generated by the union of~\({\cal F}_X\) and  the induced subgraph of~\(\Gamma_i({\cal Y})\) with vertex set~\({\cal M}_l := \{{\cal N}_j(u)\}_{0 \leq j \leq l}.\) Similarly, we let~\({\cal F}_l \supset {\cal G}_l\) be the sigma-field generated by the union of~\({\cal G}_l\) and the weight of all edges with at least one endvertex in~\({\cal M}_l.\) With the above definitions, we see that the events~\(\Lambda_l(u)\) and~\(I_l(u)\) defined in~(\ref{i_def}) are respectively~\({\cal G}_l-\) and~\({\cal F}_l-\)measurable. Similarly, the event~\(F_l(M,u)\)  is~\({\cal F}_{l-1}-\)measurable. 

Letting~\(\gamma > 0\) be arbitrary, we argue below that if~\(M = M(\varepsilon,\Delta,\gamma) > 0\) is chosen large enough, then the following property holds for every vertex~\(u \in {\cal S}\setminus {\cal Y}\) and every~\(1 \leq l \leq K-1:\) If~\({\cal S}_0 = \{u\}\) and~\({\cal S}_j, 1 \leq j \leq l\) are deterministic and mutually disjoint subsets of~\( {\cal S} \setminus {\cal Y},\) then
\begin{align}\label{elm_est_ff}
&\mathbb{P}\left(F_l^c(M,u) \mid {\cal G}_{l-1}\right) \ind\left(\Lambda_{l-1}(u) \bigcap F({\cal S},\Gamma)\right) \ind\left(J_{l-1}\right) \nonumber\\
&\;\;\;\;\;\leq\;\;\; \frac{1}{n^{\Delta+\gamma+3}}\ind\left(\Lambda_{l-1}(u) \bigcap F({\cal S},\Gamma)\right) \ind\left(J_{l-1}\right),
\end{align}
where~\[J_{l-1} := \bigcap_{j=0}^{l-1} \left\{{\cal N}_j(u) = {\cal S}_j\right\}\] is~\({\cal G}_{l-1}-\)measurable by definition,
\begin{align}\label{elm_est_ff_33}
&\mathbb{P}\left(\Psi_l^c(M,u) \mid {\cal F}_{l-1}\right) \ind\left(I_{l-1}(u) \bigcap F_l(M,u) \bigcap F({\cal S},\Gamma)\right) \ind\left(J_{l}\right) \nonumber\\
&\;\;\;\;\;\leq\;\;\; \frac{1}{n^{\Delta+\gamma+3}}\ind\left(I_{l-1}(u) \bigcap F_l(M,u) \bigcap F({\cal S},\Gamma)\right) \ind\left(J_{l}\right),
\end{align}
and
\begin{align}\label{elm_est}
&\mathbb{P}\left(E_l^c(M,u) \mid {\cal F}_{l-1}\right) \ind\left(I_{l-1}(u) \bigcap  F_l(M,u) \bigcap F({\cal S},\Gamma)\right) \ind\left(J_{l}\right) \nonumber\\
&\;\;\;\;\;\leq\;\;\; \frac{1}{n^{\Delta+\gamma+3}}\ind\left(I_{l-1}(u) \bigcap  F_l(M,u) \bigcap F({\cal S},\Gamma)\right) \ind\left(J_{l}\right),
\end{align}
where
\[J_{l} := \bigcap_{j=0}^{l} \left\{{\cal N}_j(u) = {\cal S}_j\right\} = J_{l-1} \bigcap \left\{{\cal N}_{l}(u) = {\cal S}_l\right\}\] is~\({\cal F}_{l-1}-\)measurable, since the knowledge of weights of all edges with at least one endvertex in~\({\cal N}_{l-1}(u)\) precisely determines~\({\cal N}_l(u).\)

We prove~(\ref{elm_est}),~(\ref{elm_est_ff_33}) and~(\ref{elm_est_ff}) in that order below. Because~\(I_{l-1}(u)\) occurs, it suffices to consider sets~\(\{{\cal S}_j\}_{1 \leq j \leq l-1}\) such that~\({\cal S}_j\) has at most~\(d_{up}^{j}\) vertices, for each~\(1 \leq j \leq l-1.\) Also the estimate~(\ref{zimax5}) implies that it suffices to consider only sets~\({\cal S}_l\) containing at most~\(d_{up}^{l}\) vertices.

Each edge with both endvertices in the vertex set~\({\cal V}(W_i)\) is independently present in the graph~\(\Gamma_i,\) with probability~\(p = \frac{1}{N^{1-\varepsilon}},\) as derived in~(\ref{prob_est}).  This implies that if~\(v \in {\cal R}_l(u)\) is at a distance of at least~\(l\) from~\(u\) in~\(\Gamma_i\left({\cal Y}\right)\) and~\(E_l(M,u,v)\) is the event that~\(v\) is adjacent to at most~\(M\) vertices of~\({\cal N}_l(u),\) then~\(E_l(M,u,v)\) is determined by the weights of edges containing \emph{both} endvertices in~\({\cal R}_l(u)\) and is therefore independent of~\({\cal F}_{l-1}.\) Consequently, recalling that~\({\cal N}_l(u) = {\cal S}_l\) has at most~\(d_{up}^l\) vertices, we get that
\begin{align}
&\mathbb{P}\left(E_l^c(M,u,v) \mid {\cal F}_{l-1} \right) \ind\left(I_{l-1}(u) \bigcap F_l(M,u) \bigcap  F({\cal S},\Gamma)\right) \ind\left(J_{l}\right) \nonumber\\
&\;\;\;\;\;\leq\;\;\; {d_{up}^l \choose M} p^{M} \ind\left(I_{l-1}(u) \bigcap F_l(M,u) \bigcap F({\cal S},\Gamma)\right) \ind\left(J_{l}\right). \label{tiramisu}
\end{align}

Recalling that~\(d_{up} = D_2Np\) where~\(D_2 > 0\) is a constant (see~(\ref{d_low_up_est})), we have that
\begin{align}
{d_{up}^l \choose M} \cdot p^{M} &\leq \left(d_{up}^l \cdot p\right)^M \nonumber\\
&= \left((D_2Np)^{l} \cdot p\right)^{M} \nonumber\\
&= C\left(\frac{Np}{N^{\frac{1}{l+1}}}\right)^{M(l+1)} \nonumber\\
&= C \left(\frac{1}{N^{\frac{1}{l+1}-\varepsilon}}\right)^{M(l+1)} \label{tiramisu_two}
\end{align}
for some constant~\(C > 0,\) where the final expression in~(\ref{tiramisu_two}) is true since~\(p = \frac{1}{N^{1-\varepsilon}}\) by~(\ref{prob_est}).

By choice~\(1 \leq l \leq K-1\) and~\(K = K(\varepsilon)\) is the largest  integer smaller than~\(\frac{1}{\varepsilon}.\) Since~\(\varepsilon  >0\) is irrational, we have that~\(K  < \frac{1}{\varepsilon}\) strictly and so~\[\frac{1}{l+1} \geq \frac{1}{K} > \varepsilon\] strictly. Also,~\(N = nr_n^2 \geq n^{1-2\beta} \rightarrow \infty\) by Theorem statement and  so given~\(\gamma > 0,\) we choose the constant~\(M = M(\varepsilon,\gamma,\beta,\Delta) > 0\) large enough so that \[C \left(\frac{1}{N^{\frac{1}{K}-\varepsilon}}\right)^{2M} \leq \frac{1}{n^{\Delta+\gamma+4}}\]
for all~\(n\) large, where we recall that~\(\Delta\) is the size of the extraneous set~\({\cal Y}.\) With such a choice of~\(M,\) we see that the final term in~(\ref{tiramisu_two}) is at most~\(\frac{1}{n^{\Delta+\gamma+4}}.\) The event~\(E_l(M,u)\)  is simply the intersection of all the events~\(\{E_l(M,u,v)\}_{v \in {\cal R}_l(u)}\) and since there are at most~\(n\) vertices in~\({\cal R}_l(u),\) we get the desired estimate~(\ref{elm_est}) by an application of the union bound.


To establish~(\ref{elm_est_ff_33}), we argue as in the estimation of the probability of the event~\(E_{heavy}(i)\) defined prior to~(\ref{e_nice_loc_est_two}). Indeed, because~\(I_{l-1}(u)\) occurs, we know that the total number of vertices in~\(\bigcup_{j=1}^{l-1} {\cal N}_j(u)\) is \[\#\left(\bigcup_{j=1}^{l-1} {\cal N}_j(u)\right) \leq \sum_{j=1}^{l-1} d_{up}^{j} \leq l \cdot d_{up}^{l-1}.\] From the discussion prior to~(\ref{tiramisu_two}), we recall that~\(d_{up}\) is at most of the order of~\(Np = N^{\varepsilon}\) (see~(\ref{prob_est})). Therefore, for~\(l \leq K\) we get that \[\#\left(\bigcup_{j=1}^{l-1} {\cal N}_j(u)\right) \leq D \cdot N^{\varepsilon(l-1)} \leq D \cdot N^{\varepsilon (K-1)} = o(N)\] for some constant~\(D > 0,\) since~\(\varepsilon < \frac{1}{K-1},\) by our choice of~\(K\) in~(\ref{k_chce}).

By definition the set~\({\cal Y}\) has~\(\Delta\) vertices for~\(\Delta > 0\) constant and the total number of vertices in~\(\Gamma_i({\cal Y})\) is at least~\(D_1N -\Delta\) for some constant~\(D_1 > 0\) and all~\(n\) large, due to the occurrence of~\(E_{vert} \supset I_{l-1}(u).\) Consequently, the discussion in the above paragraph implies that the number of vertices in~\({\cal R}_l(u)\) is at least~\(D_1N-\Delta - o(N),\) for all~\(n\) large. As argued prior to the property~\((p4)\) of Step~\(1\) above, the event~\(E_{short} \supset I_{l-1}(u)\) ensures that each vertex of~\({\cal R}_l(u)\) is adjacent to at least~\[D_1N-\Delta-o(N) \geq D_0N\] other vertices of~\({\cal R}_l(u),\) by \emph{long} edges; i.e., edges of length at least~\(s_n.\) Here~\(D_0\) is the constant appearing in property~\((p4).\) Therefore arguing as in the derivation of the estimate~(\ref{e_nice_loc_est_two}) for~\(E_{heavy}(i),\) we obtain
\begin{align}
&\mathbb{P}\left(\Psi_l^c(M,u) \mid {\cal F}_{l-1}\right) \ind(I_{l-1}(u) \bigcap F_l(M,u) \bigcap F({\cal S},\Gamma)) \ind( J_l) \nonumber\\
&\;\;\leq\;\;\exp\left(-DN^{\varepsilon}\right) \ind(I_{l-1}(u) \bigcap F_l(M,u) \bigcap F({\cal S},\Gamma)) \ind( J_l) \nonumber\\
&\;\;\leq\;\;\frac{1}{n^{\Delta+\gamma+3}} \ind(I_{l-1}(u) \bigcap F_l(M,u) \bigcap F({\cal S},\Gamma)) \ind( J_l) \nonumber
\end{align}
for some constant~\(D >0\) and all~\(n\) large. This obtains~(\ref{elm_est_ff_33}).

Finally, to estimate the probability of the event~\(F_l(M,u),\) we use the fact that~\(\Lambda_{l-1}(u)\) occurs and so the number of vertices in~\({\cal N}_{l-1}(u)\) is at most~\(d_{up}^{l-1} \leq d_{up}^{l}.\) Therefore arguing as in the derivation of~(\ref{elm_est}), we get~(\ref{elm_est_ff}).  Taking averages in~(\ref{elm_est}) and summing over all the possible sets~\(\{{\cal S}_j\},\) we get that
\begin{align}
&\mathbb{P}\left(E_l^c(M,u) \bigcap F_l(M,u) \bigcap I_{l-1}(u) \bigcap F({\cal S},\Gamma)\right) \nonumber\\
&\;\;\leq\;\;\frac{1}{n^{\Delta+\gamma+3}} \mathbb{P}\left(F_l(M,u) \bigcap I_{l-1}(u) \bigcap F({\cal S}, \Gamma)\right) \nonumber\\
&\;\;\leq\;\;\frac{1}{n^{\Delta+\gamma+3}} \mathbb{P}\left(F({\cal S},\Gamma)\right). \nonumber
\end{align}
This obtains~(\ref{el_est_app}).

Similarly taking averages in~(\ref{elm_est_ff_33}) summing over all possible sets~\(\{{\cal S}_j\},\) we get
\begin{equation}
\mathbb{P}\left(\Psi_l^c(M,u) \bigcap F_l(M,u) \bigcap I_{l-1}(u) \bigcap F({\cal S},\Gamma)\right) \leq \frac{1}{n^{\Delta+\gamma+3}} \mathbb{P}\left(F({\cal S},\Gamma)\right). \nonumber
\end{equation}
This obtains~(\ref{psi_l_est_app}). Finally, since~\(I_{l-1}(u) \subset \Lambda_{l-1}(u)\) (see~(\ref{i_def})), we again sum over all possible~\(\{{\cal S}_j\}\) to get that
\begin{align}
&\mathbb{P}\left(F_l^c(M,u) \bigcap I_{l-1}(u) \bigcap F({\cal S},\Gamma)\right) \nonumber\\
&\;\; \leq \;\;\;\mathbb{P}\left(F_l^c(M,u) \bigcap I_{l-1}(u) \bigcap F({\cal S},\Gamma)\right) \nonumber\\
&\;\;\leq\;\;\frac{1}{n^{\Delta+\gamma+3}} \mathbb{P}\left(F({\cal S},\Gamma)\right). \nonumber
\end{align}
This obtains~(\ref{fl_est_app}) and therefore completes the proof of~(\ref{h_gen_est}).

We now use~(\ref{h_gen_est}) to derive the desired estimate~(\ref{diam_est_am}). If~\(I_{K-1}(u)\) occurs, then we know that
\begin{equation}\label{nku_est}
\#{\cal N}_{K-1}(u) \geq \frac{1}{M^{K-2}}d_{low}^{K-1} = 2C_0 (Np)^{K-1}
\end{equation}
for some constant~\(C_0 > 0,\) where the final estimate in~(\ref{nku_est}) is true since~\(d_{low} = \frac{D_0Np}{2}\) for some constant~\(D_0 > 0\) by definition (see~(\ref{d_low_up_est})).

Below we describe how the lower bound in~(\ref{nku_est}) facilitates in showing that every remaining vertex of~\(\Gamma_{i}\left({\cal Y}\right)\) is adjacent to at least one vertex of~\({\cal N}_{K-1}(u).\) We recall from~(\ref{rl_def}) that~\({\cal R}_l(u)\) is the set of all vertices at a distance of at least~\(l\) from~\(u\) and we let~\(I_K(u)\) be the event that every vertex of~\({\cal R}_{K}(u)\) is adjacent to at least one vertex of~\({\cal N}_{K-1}(u).\) To estimate the probability of~\(I_K(u),\) we again let~\({\cal S}_j, 1 \leq j \leq K-1\) be deterministic subsets of~\(\{1,2,\ldots,n\}\) as in~(\ref{elm_est}) and assume that the event \[J_{K-1} := \bigcap_{j=1}^{K-1} \left\{{\cal N}_j(u) = {\cal S}_j\right\}\] defined in~(\ref{elm_est_ff}) occurs. Also because the event~\(I_{K-1}(u)\) as defined in~(\ref{i_def})  also involves the event~\(Q_{K-1}(M,u)\) related to the state of \emph{all} edges having at least one endvertex in~\({\cal N}_{K-1}(u),\) we use the event
\begin{equation}\label{i_s_def}
\Lambda_{K-1}(u) =  E_{vert} \bigcap E_{short} \bigcap H_{K-1}(u) \supset I_{K-1}(u)
\end{equation}
defined in~(\ref{i_def}), for future purposes. Recalling that~\({\cal G}_{K-1}\)  is the sigma-field generated by the vertex locations and the induced subgraph of~\(\Gamma_i\) with vertex set~\(\{{\cal N}_j(u)\}_{1 \leq j \leq K-1},\) we get that the events~\(F({\cal S},\Gamma),\Lambda_{K-1}(u)\) and~\(J_{K-1}\) are all~\({\cal G}_{K-1}-\)measurable.

The advantage of using~\(\Lambda_{K-1}(u)\) is that given  the occurrence of the joint event~\[\Lambda_{K-1}(u) \cap J_{K-1}  \cap F({\cal S},\Gamma),\] we get from the estimate~(\ref{nku_est}) that any vertex~\(v \in {\cal R}_{K}(u)\) is \emph{not} adjacent to any vertex of~\({\cal N}_{K-1}(u)\) with conditional probability
\begin{equation}\label{t_aaaa}
T_v := (1-p)^{\#{\cal N}_{K-1}(u)-N_{short}(v)},
\end{equation}
where~\(N_{short}(v)\) is the number of short edges (i.e., edges of length at most~\(q_n\)), containing~\(v\) as an endvertex in~\(\Gamma_i.\) Recalling the estimate~(\ref{nku_est}) and the corresponding constant~\(C_0 > 0,\) we have that
\begin{align}\label{tal_bn}
\#{\cal N}_{K-1}(u) &\geq 2C_0 (Np)^{K-1} \nonumber\\
&= 2C_0 N^{\varepsilon (K-1)}   \nonumber\\
&= 2C_0 (nr_n^2)^{\varepsilon (K-1)}  \nonumber\\
&\geq 2C_0 n^{(1-2\beta)\varepsilon (K-1)},
\end{align}
where the first equality in~(\ref{tal_bn}) is true since~\(p = \frac{1}{N^{1-\varepsilon}}\) by~(\ref{prob_est}) and the final bound in~(\ref{tal_bn}) follows from the fact that~\(r_n \geq \frac{1}{n^{\beta}},\) by Theorem statement. The bound~(\ref{tal_bn}) is consistent since~\(\varepsilon (K-1) < 1\) by the statement prior to~(\ref{k_chce}) and so~\(N^{\varepsilon (K-1)} = o(N)\) is much smaller than order of~\(N,\) the total number of vertices in~\(\Gamma_i.\)

Since~\(\Lambda_{K-1}(u) \subset E_{short} \in {\cal G}_{K-1}\) occurs, we get from the estimate~(\ref{shortt_edges}) that there is a constant~\(D > 0\) such that
\begin{align}\label{tal_bn2}
N_{short}(v) &\leq D nq_n^2  \nonumber\\
&= Dn^{1-2c} \nonumber\\
&= o\left(n^{(1-2\beta)\varepsilon (K-1)}\right),
\end{align}
where the first equality in~(\ref{tal_bn2}) is true since~\(q_n = \frac{1}{n^{c}},\) by Theorem statement and the final estimate in~(\ref{tal_bn}) follows from our choice of~\(\varepsilon\) and~\(K\) in~(\ref{k_chce}).

Combining~(\ref{tal_bn2}) with~(\ref{tal_bn}) we get that~\(N_{short}(u)\) is much smaller than~\((Np)^{K-1}\) and so~(\ref{t_aaaa}) implies that \[T_v \leq \exp\left(-C_0 p(Np)^{K-1}\right),\] or equivalently,
\begin{align}
&\mathbb{P}\left(v \notin  {\cal N}_{K}(u) \mid {\cal G}_{K-1}\right) \ind\left(\Lambda_{K-1}(u) \bigcap F({\cal S},\Gamma)\right) \ind\left(J_{K-1}\right) \nonumber\\
&\;\;\;\;\;\leq\;\;\; \exp\left(-C_0 p(Np)^{K-1}\right)\ind\left(\Lambda_{K-1}(u)  \bigcap F({\cal S},\Gamma)\right) \ind\left(J_{K-1}\right). \nonumber
\end{align}
An application of the union bound then gives
\begin{align}
&\mathbb{P}\left(I^c_K(u) \mid {\cal G}_{K-1}\right) \ind\left(\Lambda_{K-1}(u) \bigcap F({\cal S},\Gamma)\right) \ind\left(J_{K-1}\right) \nonumber\\
&\;\;\;\;\;\leq\;\;\; n\cdot \exp\left(-C_0 p(Np)^{K-1}\right)\ind\left(\Lambda_{K-1}(u)  \bigcap F({\cal S},\Gamma)\right) \ind\left(J_{K-1}\right) \nonumber
\end{align}
and so summing  over  the sets~\(\{{\cal S}_j\}_{1 \leq j \leq K-1},\) we find that
\begin{align}
&\mathbb{P}\left(I_K^c(u) \bigcap \Lambda_{K-1}(u) \bigcap F({\cal S},\Gamma)\right)  \nonumber\\
&\;\;\;\;\;\leq\;\;\; n \exp\left(-C_0 p(Np)^{K-1}\right)\mathbb{P}\left(\Lambda_{K-1} \bigcap F({\cal S},\Gamma)\right) \nonumber\\
&\;\;\;\;\;\leq\;\;\; n \exp\left(-C_0 p(Np)^{K-1}\right)\mathbb{P}\left(F({\cal S},\Gamma)\right). \label{telly}
\end{align}

Since~\(p = \frac{1}{N^{1-\varepsilon}}\) (see~(\ref{prob_est})), we get that
\[p(Np)^{K-1} = N^{K-1} \cdot p^{K} = \left(N^{\frac{K-1}{K}} p\right)^{K} = \left(N^{\varepsilon - \frac{1}{K}}\right)^{K}.\]
By choice~\(K < \frac{1}{\varepsilon}\) strictly  (see discussion preceding~(\ref{k_chce})) and so given~\(\gamma > 0,\) we get from~(\ref{telly}) that
\begin{equation}
\mathbb{P}\left(I_K^c(u) \bigcap \Lambda_{K-1}(u) \bigcap F({\cal S},\Gamma)\right) \leq \frac{1}{n^{\Delta+\gamma+3}}\mathbb{P}\left(F({\cal S},\Gamma)\right),\label{telly2}
\end{equation}
where we recall that~\(\Delta\) is the size of the extraneous set~\({\cal Y}.\)

We now combine~(\ref{telly2}) and the iteration estimate~(\ref{h_gen_est}) and using telescoping to estimate the conditional probability of occurrence of~\(I_K(u).\) Indeed, letting~\[\mathbb{P}_{F}(.) := \mathbb{P}\left(. \mid F({\cal S},\Gamma)\right)\] be the conditional distribution, we have that
\begin{align}
\mathbb{P}_F\left(I^c_K(u)\right) &= \mathbb{P}_F\left(I^c_K(u) \cap \Lambda_{K-1}(u)\right) + \mathbb{P}_F\left(I^c_K(u) \cap \Lambda^c_{K-1}(u) \right) \nonumber\\
&\leq \frac{1}{n^{\Delta+\gamma+3}}  + \mathbb{P}_F\left(I^c_K(u) \cap \Lambda^c_{K-1}(u)\right) \nonumber\\
&\leq \frac{1}{n^{\Delta+\gamma+3}}  + \mathbb{P}_F\left(\Lambda^c_{K-1}(u)\right) \nonumber\\
&\leq \frac{1}{n^{\Delta+\gamma+3}}  + \mathbb{P}_F\left(I^c_{K-1}(u)\right)\label{eve_tits}
\end{align}
where the first inequality in~(\ref{eve_tits}) is due to~(\ref{telly2}) and the final estimate in~(\ref{eve_tits}) is true because~\(I_{K-1}(u) \subset \Lambda_{K-1}(u)\) by definition (see~(\ref{i_def})). Proceeding iteratively and using~(\ref{h_gen_est}), we then get that
\begin{equation} \label{kalam}
\mathbb{P}_F\left(I^c_{K}(u)\right) \leq \frac{K}{n^{\Delta+\gamma+3}} + \mathbb{P}_F\left(I^c_0(u)\right).
\end{equation}
By definition, we recall that we have assumed that the event~\(F({\cal S}, \Gamma)\) occurs for deterministic~\({\cal S}\) and~\(\Gamma\) that is consistent with the occurrence of the event~\(E_{vert} \cap E_{short} = I_0(u).\) Therefore the final term in~(\ref{kalam}) is zero and so
we obtain
\begin{equation}\label{h_one_est_ax}
\mathbb{P}\left(I^c_K(u) \mid F({\cal S}, \Gamma)\right) \leq \frac{K}{n^{\Delta+\gamma+3}}
\end{equation}
for each~\(u \in {\cal V}(W_i)\setminus {\cal Y} = {\cal S} \setminus {\cal Y}.\)

In words, if~\(I_K(u)\) occurs, then each vertex in~\(\Gamma_i({\cal Y})\) is at a distance of at most~\(K\) from~\(u.\) Therefore defining
\[E_{diam}(W_i) := \bigcap_{{\cal Y}} \bigcap_{u \in {\cal V}(W_i) \setminus {\cal Y}} I_K(u)\] and using the fact that there are at most~\(n\) vertices in~\({\cal V}(W_i)\) and at most~\(n^{\Delta}\) choices for~\({\cal Y},\) we get from the union bound that
\begin{equation} \nonumber
\mathbb{P}\left(E_{diam}(W_i) \mid F({\cal S}, \Gamma)\right) \geq 1-\frac{K}{n^{2+\gamma}}
\end{equation}
and averaging over the feasible values of~\(({\cal S},\Gamma),\) we have
\begin{equation} \label{salpa}
\mathbb{P}\left(E_{diam}(W_i)\right) \geq 1-\frac{K}{n^{2+\gamma}}.
\end{equation}

We recall that there are~\(T = \frac{4\zeta_n}{r_n}\) squares~\(\{W_i\}_{1 \leq i \leq T}\) in total and~\(T \leq 2n\) by~(\ref{t_est}). Therefore, defining \begin{equation}\label{e_diam_def}
E_{diam} := \bigcap_{i=1}^{T} E_{diam}(W_i),
\end{equation}
we get from~(\ref{salpa}) and the union bound that
\begin{equation}\label{e_diam_est}
\mathbb{P}(E_{diam}) \geq 1-\frac{K}{n^{1+\gamma}}.
\end{equation}
If~\(E_{diam}\) occurs, then for any~\({\cal Y}\) containing~\(\Delta\) vertices, the diameter of~\(\Gamma_i({\cal Y})\) is at most~\(K.\) This implies that~\(F_{diam}(K,\Delta)\) occurs and  since~\(\gamma > 0\) is arbitrary and~\(K\) is a constant, this obtains~(\ref{diam_est_am}) and therefore completes the proof of Step~\(2.\)

\emph{\underline{Step 3}}: In this final step, we use the events described in Steps~\((1)-(2)\) to construct a path from~\(O_1\) to~\(O_2,\) that satisfies the properties described in the Theorem statement. Formally, for constants~\(K,\Delta > 0,\) we recall the events~\(E_{tot}\) and~\(F_{diam}(K,\Delta)\) defined respectively in Steps~\(1\) and~\(2\) above and assume that the joint event~\[E_{tot} \cap F_{diam}(K,K+1)\] occurs. Given~\(\gamma > 0,\) we invoke the union bound and get from the respective probability estimates~(\ref{e_tot_est}) and~(\ref{diam_est_am}) that
\begin{equation}\label{f_joint}
\mathbb{P}\left(E_{tot} \cap F_{diam}(K,K+1)\right) = 1-\frac{2}{n^{1+\gamma}}-e^{-\theta_0 L_n}
\end{equation}
for some constant~\(\theta_0 > 0.\)

We recall the~\(\frac{r_n}{4} \times \frac{r_n}{4}\) squares~\(\{W_i\}_{1 \leq i \leq T}\) ``between" the vertices~\(O_1\) and~\(O_2\) defined prior to Step~\(1\) and identify those squares that contain a super nice edge; i.e., an edge with weight at least~\(H(N^2)\) having both its endvertices in the same square~\(W_j.\)  For convenience, we recall that~\(N = nr_n^2\) and~\(H(.)\) is the inverse ccdf as defined in~(\ref{h_def}). Because the event~\(E_{lambda} \supset E_{tot}\) described prior to~(\ref{tasker}) occurs, we know that there are at least~\(\theta T\) distinct squares of~\(\{W_i\},\) each containing a super nice edge. Let~\({\cal I} \subset \{1,2,\ldots,T\}\) be~\(\theta T\) indices such that each~\(W_j, j \in {\cal I},\) contains a super nice edge~\(h_j = (c_j,d_j)\) having both its endvertices~\(c_j\) and~\(d_j\) in~\(W_j.\)

We now ``stitch" together these super nice edges using nice edges to obtain a path of sufficiently large weight, as follows. Since~\(E_{cross} \supset E_{tot}\) occurs, for each~\(1 \leq i \leq T-1,\) there is a nice edge~\(f_i = (a_i,b_i)\) having one endvertex~\(a_i \in W_i\) and the other endvertex~\(b_i \in W_{i+1}.\) Let~\({\cal P}_1\) be a nice path  in~\(W_1\) (i.e., a path consisting of only of nice edges, each edge having both its endvertices in~\(W_1\)) of length at most~\(K,\) and having endvertices~\(x_0 := O_1 = (0,0)\) and~\(x_1 := a_1 \in W_1.\) The existence of such a path is guaranteed due to the occurrence of~\(F_{diam}(K,K+1).\) Similarly, let~\({\cal P}_T\) be a nice path of length at most~\(K\) in~\(W_T,\) having endvertices~\(x_{T-1} := b_{T-1} \in W_T\) and~\(x_T  := O_2 = (\zeta_n,0).\)

For each~\(2 \leq i \leq T-1,\) we construct a nice path~\({\cal P}_i\) of length at most~\(2K+1\) in~\(W_i,\) with endvertices~\(x_{i-1} := b_{i-1} \in W_i\) and~\(x_i := a_i \in W_i,\) as follows. If~\(i \notin {\cal I},\) then we let~\({\cal P}_i\) be any nice path of length at most~\(K\) in~\(W_i,\) having endvertices~\(x_{i-1}\) and~\(x_i.\) On the other hand, if~\(i \in {\cal I}\) then~\(W_i\) has a super nice edge~\(h_i = (c_i,d_i).\) We therefore let~\({\cal P}_{i,1}\) be a  nice path of length at most~\(K\) in~\(W_i\) having endvertices~\(x_{i-1}\) and~\(c_i.\) We then pick a nice path~\({\cal P}_{i,2}\) of length at most~\(K,\) with endvertices~\(d_i\) and~\(x_i\) and having \emph{no} vertex in common with~\({\cal P}_{i,1}.\) This is possible since~\({\cal P}_{i,1}\) has at most~\(K+1\) vertices and so the existence of~\({\cal P}_{i,2}\) is again guaranteed due to the occurrence of the event~\(F_{diam}(K,K+1).\) We then set~\[{\cal P}_i := {\cal P}_{i,1} \cup \{h_i\} \cup {\cal P}_{i,2}\] so that~\({\cal P}_i\) has length at most~\(2K+1\) and contains the super nice edge~\(h_i.\)

The above construction is illustrated in Figure~\ref{fig_path} where~\(B = W_j, j \in {\cal I},\) contains a super nice edge~\(h\) and the adjacent squares~\(A= W_{j-1}\) and~\(C = W_{j+1}\) do not contain any super nice edges. The edges~\(d,e,f\) and~\(g\) are nice cross edges that are used to stitch the paths~\(\{{\cal P}_j\}\) as described in the previous paragraph. Indeed~\({\cal P}_{j-1}\) is denoted by the curved line joining edges~\(d\) and~\(e\) and contained in the square~\(A.\) Similarly~\({\cal P}_{j+1}\) is the curved line joining~\(f\) and~\(g\) and contained within~\(C.\) The path~\({\cal P}_j = {\cal P}_{j,1} \cup \{h\} \cup {\cal P}_{j,2},\) where~\({\cal P}_{j,1}\) is the curved line joining~\(e\) and~\(h\) and~\({\cal P}_{j,2}\) is the curved line joining~\(h\) and~\(f,\) as shown in Figure~\ref{fig_path}.

\begin{figure}[tbp]
\centering
\includegraphics[width=4in, trim= 40 500 10 70, clip=true]{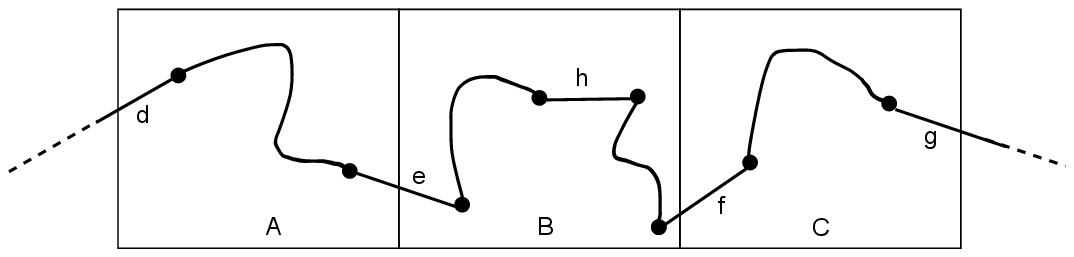}
\caption{Constructing the path~\({\cal P}_{tot}\)  using nice and super nice edges.}
\label{fig_path}
\end{figure}

By construction, the paths~\({\cal P}_i\) and~\({\cal P}_j\) are vertex disjoint if~\(|i-j| \geq 2\) and for any~\(1 \leq i \leq T-1,\) the paths~\({\cal P}_i\) and~\({\cal P}_{i+1}\) share a single vertex in common. The union
\[{\cal P}_{tot} := \bigcup_{i=1}^{T} {\cal P}_i\] is therefore a path containing~\(O_1\) and~\(O_2\) as endvertices and has the following properties:\\
\((q1)\) Each edge in~\({\cal P}_{tot}\) is nice, i.e., has length at least~\(q_n\) and weight at least~\(w_n = H(N^{1-\varepsilon}),\) where~\(N = nr_n^2,\)\\
\((q2)\) There are at least~\(\theta T\) super nice edges in~\({\cal P}_{tot}\) each having weight at least~\(H(N^2)\) and\\
\((q3)\) Each~\({\cal P}_i\) has at most~\(2K+1\) edges.\\
Property~\((q2)\) immediately implies that the total weight of the edges in~\({\cal P}_{tot}\) is at least
\[\theta T \cdot H(N^2) = \frac{4\theta \zeta_n}{r_n} H(N^2), \] where the expression~\(T = \frac{4\zeta_n}{r_n}\) follows from~(\ref{t_est}). Property~\((q3)\) ensures that the total length of~\({\cal P}_{tot}\) is at most~\(T(2K+1) = \frac{4(2K+1)\zeta_n}{r_n}.\)

Summarizing, we see that if the event~\(E_{tot} \cap F_{diam}(K,K+1)\) occurs, then
\[M_n\left(\lambda_1 L_n, q_n, w_n\right) \geq \lambda_2 L_n \alpha_n\] where~\(L_n=\frac{\zeta_n}{r_n}\) is as in Theorem statement and~\(\lambda_1 = 4(2K+1)\) and~\(\lambda_2 = 4\theta\) are constants. The probability estimate~(\ref{f_joint}) then obtains the desired deviation bound~(\ref{weight_path_low}) in Theorem statement and therefore completes the proof of the Theorem.~\(\qed\)

\renewcommand{\theequation}{\thesection.\arabic{equation}}
\setcounter{equation}{0}
\section{Proof of Theorem~\ref{thm_path2}} \label{sec_proof_b}
Let~\({\cal P}_{up}\) denote the maximum weight path in the random graph~\(G,\) containing at most~\(\lambda L = \lambda L_n = \frac{\lambda \zeta_n}{r_n}\) edges. Each edge of~\({\cal P}_{up}\) has length at most~\(r_n\) and so the path~\({\cal P}_{up}\) is necessarily contained in the~\(\lambda \zeta_n \times \lambda \zeta_n\) square~\(S(\zeta_n),\) centred at the origin. Let~\(G(\zeta_n)\) be the induced subgraph of~\(G\) with vertex set being the set of all vertices of~\(\{X_l\},\) present in~\(G(\zeta_n).\) In what follows, we therefore consider only edges located in~\(G(\zeta_n);\) i.e., edges having both endvertices in~\(G(\zeta_n).\)

The weight of~\({\cal P}_{up}\) equals~\(M_n = M_n(\lambda L_n,0,0)\) and we first consider the case~\(L_n \geq \chi_n := (\log{n})^{1+\kappa},\) where~\(\kappa > 0\) is as in Theorem statement. Indeed,  we use segmentation to obtain an upper bound for the weight~\(M_n\) as follows. Formally, let~\(x_0 > 0\) be the constant in the scaling relation~(\ref{f_scale}) and set
\begin{equation}\label{thet_def}
\theta_0 := \frac{\max(x_0,H(2))}{H(2)},
\end{equation}
where~\(H(.)\) is the inverse edge weight ccdf defined in~(\ref{h_def}) and so~\(H(2) >0\) by definition. For~\(j \geq 0,\) say that an edge~\((u,v)\) of~\(G(\zeta_n)\) is~\(j-\)\emph{bad}  if its weight \[W(u,v) \in [2\theta_0 j \alpha_n,  2\theta_0 (j+1)\alpha_n),\] where~\(\alpha_n = H(N^2)\) and~\(N = nr_n^2\) are as in Theorem statement.  If~\(N_{bad}(j,{\cal P}_{up})\) denotes the total number of~\(j-\)bad edges in~\({\cal P}_{up},\) then
\begin{equation}
M_n \leq 2\theta_0\sum_{j \geq 1} N_{bad}\left(j, {\cal P}_{up}\right) (j+1) \cdot \alpha_n. \label{mn_segment}
\end{equation}

Clearly, the number of~\(0-\)bad edges in~\({\cal P}_{up}\) is at most the total number of edges in~\({\cal P}_{up}\) and so
\begin{equation}\label{n_bad_zero}
N_{bad}\left(0, {\cal P}_{up}\right) \leq L = L_n
\end{equation}
To estimate~\(N_{bad}\left(j, {\cal P}_{up}\right)\) for~\(j \geq 1,\) we consider three separate cases depending on whether~\(j\) is ``large", ``intermediate" or ``small" as described below.

\underline{\emph{Case I} (Large~\(j\))}: Let~\[j_{up} := (L\log{n})^{2/s} = \left(\frac{\zeta_n \log{n}}{r_n}\right)^{2/s}\] and define~\(E_{up}\) to be the event that no edge of~\(G\) is~\(j-\)bad for some~\(j \geq j_{up}.\) We argue below that
\begin{equation}\label{j_up_est}
j_{up} \leq n
\end{equation}
and that~\(E_{up}\) occurs with high probability, i.e., with probability~\(1-o(1).\)

Indeed, since~\[\zeta_n \leq 1\;\; \text{ and }\;\;r_n \geq \frac{1}{n^{\beta}}\] for some~\(0 < \beta < \frac{1}{2},\) we get that
\[j_{up} \leq \left(\frac{\log{n}}{r_n}\right)^{2/s} = \left(n^{\beta} \log{n}\right)^{2/s} = o(n),\] using the fact that~\(s > 2\) strictly by Theorem statement and so~\(\frac{2\beta}{s} < \frac{2}{s} < 1.\) This proves~(\ref{j_up_est}).

To estimate the probability of occurrence of~\(E_{up},\) we first count the number of vertices located in~\(G(\zeta_n).\) Indeed, any vertex~\(X_j\) is independently located in the~\(\lambda \zeta_n \times \lambda \zeta_n\) square~\(S(\zeta_n),\) with probability at most
\[\int_{S(\zeta_n)}f \leq \epsilon_2 \lambda^2\zeta_n^2,\] by the density upper bound in~(\ref{f_eq}). Therefore if~\(N(\zeta_n)\) is the number of vertices of~\(\{X_j\}_{1 \leq j \leq n}\) located in~\(S(\zeta_n),\) then~\(N(\zeta_n)\) is stochastically dominated from above by a Binomial random variable with parameters~\(n\) and~\(\epsilon_2 \lambda^2 \zeta_n^2.\) Consequently, defining
\[E_{num} := \left\{N(\zeta_n) \leq 2\epsilon_2 \lambda^2 \zeta_n^2\right\},\] we get from the deviation estimate~(\ref{conc_est_f}) that
\begin{equation}\label{e_num_est}
\mathbb{P}(E_{num}) \geq 1-\exp\left(-D_0 n\zeta_n^2\right) \geq 1-\exp\left(-D_0 N\right),
\end{equation}
for some constant~\(D_0 > 0,\) where~\(N = nr_n^2\) is as in Theorem statement and the final estimate in~(\ref{e_num_est}) is true since~\(\zeta_n \geq r_n,\) by Theorem statement.

Let~\(E_{ball}\) be the event defined in the proof of Lemma~\ref{lem_animal} that ensures that each vertex  is adjacent to at most~\(D N\) other vertices of~\(G\) (and hence~\(G(\zeta_n)\)), where~\(D > 0\) is a constant. Defining~\[E := E_{num} \cap E_{ball},\] we get from the respective estimates~(\ref{e_num_est}) and~(\ref{e_ball_est}) that
\begin{equation}\label{e_ax_est_ax}
\mathbb{P}(E) \geq 1-e^{-D_0N}-e^{-C_1N } \geq 1-e^{-D_1N}
\end{equation}
for all~\(n\) large and some constant~\(D_1 > 0,\) since~\(N = nr_n^2 \geq n^{1-2\beta} \rightarrow \infty,\) by Theorem statement. Thus~\(E\) occurs with high probability and we let~\[\mathbb{P}_E(.) := \mathbb{P}\left(. \mid E\right)\] be the probability distribution conditioned on the occurrence of~\(E.\)

By standard handshaking arguments, we know that the sum of vertex degrees is twice the number of edges in any graph and  so given that~\(E\) occurs, the number of edges in~\(G(\zeta_n)\) is at most~\(D_2 \zeta_n^2nN,\) for some constant~\(D_2 > 0.\) Any~\(j-\)bad edge of~\(G\) has weight at least~\(2\theta_0j\alpha_n\) and so recalling that~\(F_c(.)\) denotes the edge weight ccdf (see discussion prior to~(\ref{h_def})), we get that any edge of~\(G(\zeta_n)\) is~\(j-\)bad with probability at most~\(F_c(2\theta_0j\alpha_n).\)

By definition (see~(\ref{thet_def})), we have that
\[\theta_0\alpha_n \geq \frac{x_0H(N^2)}{H(2)} \geq x_0\] for all~\(n\) large, since~\(N = nr_n^2 \geq n^{1-2\beta} \rightarrow \infty\) by Theorem statement and~\(H(.)\) is non-decreasing by definition. Therefore  using the scaling relation~(\ref{f_scale}),  we get that any edge of~\(G(\zeta_n)\) is~\(j-\)bad with probability
\begin{equation}
F_c(2\theta_0j\alpha_n)  \leq \frac{C_0}{j^{s}} \cdot F_c\left(2\theta_0\alpha_n\right) \nonumber
\end{equation}
Again using  the definition~(\ref{thet_def}), we see that~\(\theta_0 \geq 1\) and so we have that
\begin{equation}
F_c(2\theta_0\alpha_n) \leq F_c\left(2\alpha_n\right) = F_c(2H(N^2)) \leq \frac{1}{N^2}, \nonumber
\end{equation}
by the definition of~\(H(.)\) in~(\ref{h_def}).
Combining the above two estimates we have
\begin{equation}
F_c(2\theta_0j\alpha_n) \leq \frac{C_0}{j^s \cdot N^2}. \label{a_one}
\end{equation}

The estimate~(\ref{a_one}) implies that any edge of~\(G(\zeta_n)\) is~\(j-\)bad for some~\(j \geq j_{up} = \left(\frac{\zeta_n \log{n}}{r_n}\right)^{2/s}\) with probability at most
\[\frac{C_0}{j_{up}^s \cdot N^2} = \frac{C_0 r_n^2}{\zeta_n^2 (\log{n})^2 N^2} = \frac{C_0 }{(\log{n})^2 \zeta_n^2 nN}.\] We recall that the total number of edges in~\(G(\zeta_n)\) is at most~\(D_2 \zeta_n^2 nN\) for some constant~\(D_2 > 0\) and so recalling that~\(E_{up}\) is the event that no edge of~\(G(\zeta_n)\) is~\(j-\)bad for some~\(j \geq j_{up},\)  we get from the union bound that
\begin{equation}
\mathbb{P}_E\left(E^c_{up}\right) \leq \frac{C_0D_2}{(\log{n})^2}. \label{e_up_est_ax}
\end{equation}
Combining~(\ref{e_up_est_ax}) with the estimate~(\ref{e_ax_est_ax}) for~\(E\) and using the relation~(\ref{pab_rel_two})
\[\mathbb{P}(F) \leq \mathbb{P}_A(F) + \mathbb{P}(A^c)\]
with~\(F = E_{up}\) and~\(A = E,\) we get that
\begin{equation}\label{e_up_est}
\mathbb{P}(E_{up}^c) \leq \frac{C_0D_2}{ (\log{n})^2} + e^{-D_1 N} \leq \frac{2C_0D_2}{(\log{n})^2},
\end{equation}
for all~\(n\) large, since~\(N = nr_n^2 \geq n^{1-2\beta} \rightarrow \infty,\) by Theorem statement.

Thus~\(E_{up}\) occurs with high probability and if~\(E_{up}\) occurs, then the maximum weight path~\({\cal P}_{up}\) has no~\(j-\)bad edges for~\(j \geq j_{up}.\) This completes the analysis for the Case~\((I).\)

\underline{\emph{Case II} (Intermediate~\(j\))}:  In this case, we recall that~\(\chi_n = (\log{n})^{1+\kappa}\) where~\(\kappa > 0\) is the constant in Theorem statement and set
\begin{equation}\label{j_low_def}
j_{low} := \left(\frac{L}{\chi_n}\right)^{2/s} = \left(\frac{\zeta_n}{r_n \chi_n}\right)^{2/s}
\end{equation}
and estimate the \emph{total} number~\(\sum_{j=j_{low}}^{j_{up}}N_{bad}\left(j, {\cal P}_{up}\right)\) of~\(j-\)bad edges in the maximum weight path~\({\cal P}_{up}\) for~\(j_{low} \leq j \leq j_{up} = \left(L\log{n}\right)^{2/s}.\) We have that~\(j_{low} \geq 1\) since~\(L \geq \chi_n\) by choice.

We use the direct upper bound
\begin{equation} \nonumber
N_{bad}\left(j, {\cal P}_{up}\right) \leq N_{bad}\left(j, G(\zeta_n)\right),
\end{equation}
the total number of~\(j-\)bad edges in~\(G(\zeta_n).\) To estimate~\(N_{bad}(j,G(\zeta_n)),\) suppose as before that the event~\(E\) defined in case~\((I)\) above occurs and let~\(\mathbb{P}_E(.) := \mathbb{P}\left(. \mid E\right)\) be the probability distribution conditioned on the occurrence of~\(E.\) From the discussion prior to~(\ref{a_one}), we know that the number of edges in~\(G(\zeta_n)\) is at most~\(D \zeta_n^2 nN,\) for some constant~\(D > 0\) and from~(\ref{a_one}), we get that each edge has weight at least~\(2\theta_0 j_{low} \alpha_n\) (and hence~\(j-\)bad for some~\(j \geq j_{low}\)) with probability at most
\begin{equation}\label{p_mid_def}
\frac{C_0}{j_{low}^s \cdot N^2} = \frac{C_0r_n^2 \chi_n^2}{\lambda^2\zeta_n^2N^2}  = \frac{C_1 \chi_n^2}{\zeta_n^2\cdot nN} =:p_{mid},
\end{equation}
for some constant~\(C_1 > 0,\) since~\(N = nr_n^2.\) All constants henceforth do not depend on~\(M.\)

Consequently,~\[N_{low} := \sum_{j=j_{low}}^{j_{up}} N_{bad}(j,G(\zeta_n))\] is stochastically dominated from above by a Binomial random variable with parameters~\(D\zeta_n^2 nN\) and~\(p_{mid}\) and so
\begin{align}
\mathbb{E}_E\left(\exp\left(N_{low}\right) \right) &\leq \left(1-p_{mid} + ep_{mid}\right)^{D\zeta_n^2nN} \nonumber\\
&\leq \exp\left((e-1)D\zeta_n^2nN p_{mid}\right) \nonumber\\
&= \exp\left(D_2\chi_n^2\right),
\end{align}
for some constant~\(D_2 > 0,\) by~(\ref{p_mid_def}). Standard Chernoff bound then implies that
\begin{align}
\mathbb{P}_E\left(N_{low} \geq 2D_2 \chi_n^2\right) &\leq \exp\left(-D_2 \chi_n^2\right). \label{monica}
\end{align}

Setting
\[E_{mid} := \left\{N_{low} \leq 2D_2 \chi_n^2\right\}\] and combining~(\ref{monica}) with the estimate~(\ref{e_ax_est_ax}) for the event~\(E,\) we argue as in the derivation of~(\ref{e_up_est}) to get that
\begin{equation}\label{e_mid_est}
\mathbb{P}\left(E_{mid}\right) \leq 2\exp\left(-D_2 \chi_n^2\right),
\end{equation}
for all~\(n\) large. In other words, with high probability the number of~\(j-\)bad edges in~\(G(\zeta_n)\) for some~\(j \geq j_{low}\) (and therefore in~\({\cal P}_{up}\)) is at most~\(2D_2\chi_n^2 = 2D_2(\log{n})^{2+2\kappa}.\)

This completes the analysis of Case~\((II).\)

\underline{\emph{Case III} (Small~\(j\))}: We now consider the case~\(j \leq j_{low} = \left(\frac{L}{\chi_n}\right)^{2/s}.\)

To estimate~\(N_{bad}\left(j, {\cal P}_{up}\right)\) for~\(1 \leq j\leq j_{low},\) we use a ``multi-level" scaling argument similar to proof of Theorem~\(1\) in~\cite{kesten}. For~\(4 \leq t  = t(j,n) \leq \frac{1}{100r_n}\) to be determined later, split the unit square~\(S\) into disjoint~\(2tr_n \times 2tr_n\) squares~\(\{R_l\}_{1 \leq l \leq T_t}, T_t = \frac{1}{4t^2r_n^2},\) as described prior to the statement of Lemma~\ref{lem_animal}. Such a tiling allows for ``covering" the~\(\lambda L\) edges of the maximum weight path~\({\cal P}_{up}\) using~\(t-\)animals of size much smaller than~\(L\) and thereby facilitates the use of the deviation bound~(\ref{tom_tom}) derived in Lemma~\ref{lem_animal}.

As a first step, we show that there is a constant~\(D_{\lambda} > 0\) depending only on~\(\lambda\) and a~\(t-\)animal~\({\cal A}\) containing at most
\begin{equation}\label{q_choice}
Q := \frac{D_{\lambda} L}{t}
\end{equation}
~\(t-\)squares, with the property that every vertex of~\({\cal P}_{up}\) is contained in some square of~\({\cal A}.\) To see why this is true, we represent the path~\({\cal P}_{up}\) as~\({\cal P}_{up}  = (v_0,v_1,\ldots,v_P)\) containing~\(P \leq \lambda L\) edges~\((v_i,v_{i+1}), 0 \leq i \leq P-1.\) Let~\(P_t\) denote the integer part of~\(\frac{P}{t}\) and for~\(0 \leq i \leq P_t,\) let~\(Y_i \in \{R_l\}\) be the random~\(t-\)square containing the vertex~\(v_{it}.\)

If~\(v_m, it \leq m < (i+1)t,\) is an ``intermediate" vertex of~\({\cal P}_{up},\) then the Euclidean distance between~\(v_m\) and~\(v_{it}\) is at most~\(tr_n,\) since the length of any edge of~\(G(\zeta_n)\) is at most~\(r_n.\) This necessarily implies that either~\(v_m\) belongs to~\(Y_i\) or some~\(t-\)square sharing a corner with~\(Y_i.\) Therefore denoting~\({\cal U}(Y_i)\) to be the set of all~\(t-\)squares sharing a corner with~\(Y_i,\) including~\(Y_i,\) we see that every vertex of the path~\({\cal P}_{up}\) is contained in some~\(t-\)square of
\[{\cal A} := \bigcup_{0 \leq i \leq P_t} \left\{ {\cal U}(Y_i) \right\}.\]

The~\(t-\)squares in~\({\cal A}\) necessarily form a~\(t-\)animal since otherwise, there would be two vertices that are adjacent in~\({\cal P}_{up}\) (and hence in~\(G(\zeta_n)\)) but at a distance of at least~\(tr_n \geq 4r_n\) apart, since~\(t \geq 4\) by choice (see~(\ref{t_chic})). This leads to a contradiction since the length of any edge in~\(G(\zeta_n)\) is strictly less than~\(r_n.\) Also, for any~\(t-\)square~\(Y_i, 0 \leq i \leq P_t,\) there are at most~\(9\)~\(t-\)squares in~\(\{R_l\}\) sharing a corner with some~\(Y_i\) and so the size of the~\(t-\)animal~\({\cal A}\) is at most~\(10(P_t+1).\)

Recalling that~\(P_t\) is the integer part of~\(\frac{P}{t}\) and setting~\(C_{\lambda} := \lambda+\frac{1}{100},\) we see that
\begin{equation}\label{p_t_est}
P_t+1 \leq \frac{P}{t} + 1 \leq \frac{\lambda L}{t}+1 \leq \frac{C_{\lambda} L}{t},
\end{equation}
provided~\[t \leq (C_{\lambda}-\lambda)L = (C_{\lambda}-\lambda)\frac{\zeta_n}{r_n} = \frac{\zeta_n}{100r_n}.\] Since~\((\zeta_n,0)\) is a point of the unit square, we have that~\(0 < \zeta_n < 1\) and so combining~(\ref{p_t_est})  with the bounds for~\(t\) described in~(\ref{t_chic}), we see that our future choice of~\(t = t(j,n)\) must satisfy
\begin{equation}\label{t_choice_ax}
4 \leq t \leq \frac{\min(1,\zeta_n)}{100r_n} = \frac{\zeta_n}{100r_n} = \frac{L}{100}.
\end{equation}
For such a~\(t\), we also get from~(\ref{p_t_est}) that the size of~\({\cal A}, \) i.e., the number of~\(t-\)squares in~\({\cal A},\) is at most
\begin{equation}\nonumber
10(P_t+1) \leq \frac{10C_{\lambda}L}{t} = Q.
\end{equation}
This obtains the desired~\(t-\)animal~\({\cal A}\) as described in the paragraph containing~(\ref{q_choice}).

We now use the~\(t-\)animal~\({\cal A}\) to estimate the number of~\(j-\)bad edges~\(N_{bad}\left(j, {\cal P}_{up}\right)\) in~\({\cal P}_{up}\) by identifying~\(j-\)bad edges located in some~\(t-\)square of~\({\cal A};\) i.e.,~\(j-\)bad edges with both endvertices located in the \emph{same}~\(t-\)square of~\({\cal A}.\) Before we do so, we have the following remark. It might happen that some edge of~\({\cal P}_{up}\) has one endvertex in a~\(t-\)square~\(Y_i \in {\cal A}\) and the other endvertex in a \emph{different}~\(t-\)square~\(Y_j , j\neq i.\) To count such ``rogue" edges,   we perform an additional tiling of the unit square~\(S\) by shifting the centre of each~\(t-\)square in~\(\{R_l\}\) by~\((tr_n,tr_n).\) The same argument as above yields a shifted~\(t-\)animal~\({\cal A}_{shift}\) containing at most~\(Q = \frac{D_{\lambda} L}{t}\) shifted~\(t-\)squares such that every vertex of~\({\cal P}_{up}\) is contained in some shifted~\(t-\)square of~\({\cal A}_{shift}.\)

The advantage of including the shifted tiling is that any~\(t-\)square (shifted or not) has side length at least~\(t r_n \geq 4r_n,\) by our choice of~\(t\) (see~(\ref{t_choice_ax})). Consequently, every edge of~\({\cal P}_{up}\) is located either in some~\(t-\)square of~\({\cal A}\) or in some shifted~\(t-\)square of~\({\cal A}_{shift}.\)  Letting~\(N_{bad}\left(j, R_l\right)\) denote the total number of~\(j-\)bad edges located in the~\(t-\)square~\(R_l\) and setting~\[N_{bad}\left(j, {\cal A}\right) := \sum_{Z_i \in {\cal A}} N_{bad}\left(j, Z_i\right)\] to be the total number of~\(j-\)bad edges located in the~\(t-\)squares of~\({\cal A},\) we then get that
\begin{equation}\label{nj_bad_up}
N_{bad}\left(j, {\cal P}_{up}\right) \leq N_{bad}\left(j , {\cal A}\right) + N_{bad}\left(j , {\cal A}_{shift}\right).
\end{equation}

In what follows, we upper bound~\(N_{bad}\left(j , {\cal A}\right)\) and an analogous analysis holds for~\(N_{bad}\left(j , {\cal A}_{shift}\right)\) as well. To estimate~\(N_{bad}\left(j , {\cal A}\right),\) we appeal to Lemma~\ref{lem_animal} in Section~\ref{sec_prelim}. Indeed, by definition~\(N_{bad}\left(j, {\cal A}\right)\) is the number of~\(j-\)bad edges located in the~\(t-\)squares of the random~\(t-\)animal~\({\cal A}.\) Each~\(j-\)bad edge has weight at least~\(2\theta_0j\alpha_n,\) where we recall that~\(\theta_0 > 0\) is the constant in~(\ref{thet_def}),~\(\alpha_n = H(N^2),\;N = nr_n^2\) are as in Theorem statement and~\(H(.)\) is the inverse ccdf defined in~(\ref{h_def}). Consequently, the definition of~\(N_{edge}(.)\) prior to Lemma~\ref{lem_animal} implies that
\[N_{bad}\left(j , {\cal A}\right) \leq  N_{edge}\left(2\theta_0j\alpha_n, {\cal A}\right)\] and so setting~\(Q = \frac{D_{\lambda} L}{t},\) we use the estimate~(\ref{tom_tom}) and get for~\(y > 0\) that
\begin{equation}
\mathbb{P}\left(N_{bad}\left(j,{\cal A}\right) \geq y\right) \leq e^{\delta_1 Q \Lambda} \cdot e^{-y} + e^{-\delta_2 N}, \label{tom_tom_two}
\end{equation}
where~\(\Lambda = \max(1,t^2N^2F_c(2\theta_0j\alpha_n))\) and~\(\delta_1,\delta_2 > 0\) are constants.

From the estimate~(\ref{a_one}), we know that~\[F_c(2\theta_0j\alpha_n) \leq \frac{C_0}{j^s N^2},\] where~\(C_0 > 0\) is the constant in the scaling relation~(\ref{f_scale}). Thus the term~\(\Lambda\) in~(\ref{tom_tom_two}) is bounded above as
\[\Lambda = \max\left(1,t^2N^2F_c(2\theta_0j \alpha_n)\right) \leq \max\left(1,\frac{t^2C_0}{j^{s}}\right).\] We now set~\[t := \delta \cdot j^{s/2}\] where~\(\delta > 0\) is a large constant and see that condition~(\ref{t_choice_ax}) is satisfied since \[t = \delta \cdot j^{s/2} \leq \delta \cdot j_{low}^{s/2} = \frac{\delta L}{\chi_n} = \frac{\delta L}{(\log{n})^{1+\kappa}} \leq \frac{L}{100}\] for all~\(n\) large.  Thus
\begin{equation}\label{lambda_est}
\Lambda \leq \max(1,\delta^2C_0) \leq \delta^2 C_0,
\end{equation}
provided~\(\delta > 0\) is large enough. Fixing such a~\(\delta,\) we also get that
\begin{equation}\label{q_est}
Q = \frac{D_{\lambda} L}{t} = \frac{D_1 L}{j^{s/2}},
\end{equation}
where~\(D_{\lambda},D_1 > 0\) are constants.

Plugging~(\ref{lambda_est}) and~(\ref{q_est}) into~(\ref{tom_tom_two}), we get that
\begin{equation}
\mathbb{P}\left(N_{bad}\left(j,{\cal A}\right) \geq y\right) \leq \exp\left(\frac{D_2 L}{j^{s/2}}\right) \cdot e^{-y} + e^{-\delta_2 N}, \nonumber
\end{equation}
for some constant~\(D_2 > 0.\) Setting~\(y = \frac{2D_2 L}{j^{s/2}},\) we further have
\begin{equation}
\mathbb{P}\left(N_{bad}\left(j,{\cal A}\right) \geq \frac{2D_2 L}{j^{s/2}}\right) \leq \exp\left(-\frac{D_2 L}{j^{s/2}}\right) + e^{-\delta_2 N}. \label{tom_tom_three}
\end{equation}
Recalling that~\(j \leq j_{low} = \left(\frac{L}{\chi_n}\right)^{2/s},\) we see that~\[\frac{L}{j^{s/2}} \geq \chi_n = (\log{n})^{1+\kappa},\] where~\(\kappa > 0\) is the constant in Theorem statement. Also~\(N = nr_n^2 \geq n^{1-2\beta} \rightarrow \infty\) by Theorem statement and so we get from~(\ref{tom_tom_three}) that
\begin{equation}
\mathbb{P}\left(N_{bad}\left(j,{\cal A}\right) \geq \frac{2D_2 L}{j^{s/2}}\right) \leq 2\exp\left(-D_2 \chi_n\right) , \label{tom_tom_four}
\end{equation}
for all~\(n\) large and each~\(1 \leq j \leq j_{low}.\)

The estimate~(\ref{tom_tom_four}) obtains an upper bound for the number of~\(j-\)bad edges located in the~\(t-\)squares of the random~\(t-\)animal~\({\cal A}.\) For the shifted~\(t-\)animal~\({\cal A}_{shift},\) we use the same analysis as above. Indeed, the proof of the Lemma~\ref{lem_animal} holds for the shifted tiling as well and so choosing the constant~\(D_2 > 0\) larger and~\(D_4 > 0\) smaller, if necessary, the estimate~(\ref{tom_tom_four}) is valid for~\({\cal A}_{shift}.\) Consequently, using the upper bound~(\ref{nj_bad_up}) and the union bound, we get that
\begin{equation}
\mathbb{P}\left(N_{bad}\left(j,{\cal P}_{up}\right) \geq \frac{4D_2 L}{j^{s/2}}\right) \leq 4\exp\left(-D_2 \chi_n\right), \label{tom_tom_five}
\end{equation}
for each~\(1 \leq j \leq j_{low}.\) Further defining
\[E_{low} := \bigcap_{j=1}^{j_{low}} \left\{N_{bad}\left(j,{\cal P}_{up}\right) \leq \frac{4D_2 L}{j^{s/2}}\right\},\] another application of the union bound gives
\begin{align}
\mathbb{P}(E_{low}) &\geq 1- 4j_{low} \cdot e^{-D_2 \chi_n} \nonumber\\
&\geq 1- 4n \cdot e^{-D_2\chi_n},\label{e_low_est}
\end{align}
since~\(j_{low} \leq j_{up} \leq n,\) by~(\ref{j_up_est}). Since~\(\chi_n = (\log{n})^{1+\kappa},\) we see that~\(\mathbb{P}(E_{low})= 1-o(1)\) and this completes the analysis of the Case~\((III).\)

Finally, setting
\[E_{range} := E_{low} \bigcap E_{mid} \bigcap E_{up}\] and using the corresponding estimates~(\ref{e_low_est}),~(\ref{e_mid_est}) and~(\ref{e_up_est}), we get from the union bound that
\begin{equation}
\mathbb{P}(E_{range}) = 1-o(1). \label{e_range_est}
\end{equation}
If~\(E_{range}\) occurs, then from the corresponding definitions of the events~\(E_{low}\) and~\(E_{up},\) we see that there is a constant~\(D_{up} > 0\) such that
\begin{equation}\label{nj_badappa}
N_{bad}\left(j, {\cal P}_{up}\right)
\leq
\left\{
\begin{array}{ll}
\frac{D_{up} L}{j^{s/2}},  &\;\;\;1 \leq j \leq j_{low}\\
&\\
0, &\;\;\;j \geq j_{up}.
\end{array}
\right.
\end{equation}
Moreover, choosing~\(D_{up}\) larger if necessary, the event~\(E_{mid}\) ensures that~\[\sum_{j=j_{low}}^{j_{up}} N_{bad}(j,{\cal P}_{up}) \leq D_{up}\chi_n^2.\]

Substituting the above bounds  into the upper bound~(\ref{mn_segment}), we obtain
\begin{equation}\nonumber
M_n \leq 2\theta_0D_{up}\alpha_n \sum_{j \geq 1} \frac{L}{j^{s/2}} \cdot (j+1)  + 2\theta_0D_{up} \chi_n^2 \cdot j_{up} \cdot \alpha_n.
\end{equation}
Since~\(s > 4\) strictly by Theorem statement, we see that~\(\sum_{j \geq 1} \frac{j+1}{j^{s/2}} < \infty\) and  from the definition~(\ref{j_up_est}) we know that~\(j_{up} = \left(L \log{n}\right)^{2/s}.\) Also recalling that~\(L = \frac{\zeta_n}{r_n},\) we get that if~\(E_{range}\) occurs, then
\[M_n \leq D_{fin} L \alpha_n \left(1+ \nu_n\right)\] for some constant~\(D_{fin} > 0,\) where~\[\nu_n = \frac{\chi_n^2 (\log{n})^{2/s}}{L^{1-2/s}} = \frac{(\log{n})^{2+2\kappa+2/s}}{L^{1-2/s}},\] is as in Theorem statement. The estimate~(\ref{e_range_est}) for the event~\(E_{range}\) then obtains the desired upper bound for the maximum weight~\(M_n\) in the Theorem statement, for the case~\(L_n \geq \chi_n = (\log{n})^{1+\kappa}.\)

For the case~\(L_n \leq \chi_n,\) we let~\(\Lambda(\zeta_n)\) be the maximum weight of an edge of the graph~\(G(\zeta_n)\) and simply use the upper bound
\[M_n \leq \lambda L_n \cdot \Lambda(\zeta_n). \] To estimate~\(\Lambda(\zeta_n),\) we recall that if the event~\(E_{up}\) defined prior to~(\ref{e_up_est_ax}) holds, then~\[\Lambda(\zeta_n) \leq D \cdot j_{up} \cdot \alpha_n = D\alpha_n L^{2/s} \cdot (\log{n})^{2/s} \] for some constant~\(D > 0.\) To estimate the probability of the event~\(E_{up},\) we see that instead of the bound~(\ref{a_one}), we have
\begin{equation}
F_c(2\theta_0j\alpha_n) \leq \frac{C_0 \chi_n}{j^s \cdot N^2} \label{a_one_2}
\end{equation}
and so arguing as in the derivation of~(\ref{e_up_est}), we have
\begin{equation}\label{e_up_est_22}
\mathbb{P}(E_{up}^c) \leq \frac{D\chi_n}{(\log{n})^2} = \frac{D(\log{n})^{1+\kappa}}{(\log{n})^2},
\end{equation}
for some constant~\(D>  0.\) Choosing~\(0 < \kappa < 1,\) we obtain the desired upper bound for~\(M_n\) in Theorem statement for the case~\(L_n \leq \chi_n\)  and this completes the proof of the Theorem.~\(\qed\)

\renewcommand{\theequation}{\thesection.\arabic{equation}}
\setcounter{equation}{0}
\section{Proof of Corollary~\ref{cor_one}} \label{sec_proof_cor}
If~\(\varepsilon > 0\) is small enough, then the condition~\(q_n = r_n^{1+\varepsilon}\) in~(\ref{nice_cond_ax}) together with the estimate~(\ref{weight_path_low}) in Theorem~\ref{thm_path} implies the lower bound in~(\ref{dev_complete}). Similarly, using~\(L_n \geq (\log{n})^{b}\) for some~\(b > \frac{2(s+1)}{s-2}\)  we get that the term~\(\nu_n\) in the statement of Theorem~\ref{thm_path2} satisfies
\begin{align}
\nu_n &= \frac{(\log{n})^{2+2\kappa+2/s}}{L_n^{1-2/s}}  \nonumber\\
&\leq \frac{(\log{n})^{2+2\kappa+2/s}}{(\log{n})^{b(1-2/s)}} \nonumber\\
&= o(1), \nonumber
\end{align}
provided we set~\(\kappa  = \kappa(b)> 0\) small enough. Therefore the estimate~(\ref{weight_path_up}) in Theorem~\ref{thm_path2}  obtains the upper bound in~(\ref{dev_complete}). This completes the proof of part~\((a)\) of Corollary~\ref{cor_one}.

For the power law decay case in part~\((b),\) we use the inequalities in~(\ref{f_power}) to  see that
\[F_c(ax) \leq \frac{a_2}{a^sx^s} \leq \frac{a_2/a_1}{a^s} \cdot F_c(x)\] for all~\(x\) large. Thus the scaling condition~(\ref{f_scale}) holds. Moreover, there are constants~\(D_1,D_2 > 0\) such that
\[D_1 z^{1/s} \leq H(z) \leq D_2 z^{1/s}\] for all~\(z\) large. Thus~\[w_n = H(N^{1-\varepsilon}) \geq C_1 g_n^{1-\varepsilon}\] for some constant~\(C_1 > 0,\) where~\(g_n = (nr_n^2)^{1/s}\) is as in Theorem statement. Similarly~\(\alpha_n = H(N^2)\) satisfies~\(C_2 g_n^2 \leq \alpha_n \leq C_3 g_n^2\) for some constants~\(C_2,C_3 > 0.\) The estimate~(\ref{dev_complete}) then obtains~(\ref{example_one}) and this completes the proof of part~\((b)\) of Corollary~\ref{cor_one}.

Finally, for the exponential decay case, we see that the edge weight ccdf~\(F_c\) is continuous by choice and we now demonstrate that the scaling relation~(\ref{f_scale}) holds. Indeed for~\(x > 0, a > 1\)  and~\(s > 0\) we have that
\begin{align}
F_c(ax) &= \exp\left(-a^{\lambda} x^{\lambda}\right)  \nonumber\\
&= \exp\left(-(a^{\lambda}-1)\right) \cdot \exp\left(-x^{\lambda}\right) \nonumber\\
&\leq \frac{C}{a^{s}} \cdot \exp\left(-x^{\lambda}\right),
\end{align}
where~\(C = C(\lambda,a,s) > 0\) is a large constant. Thus~(\ref{f_scale}) is satisfied. To apply Theorems~\ref{thm_path}-~\ref{thm_path2}, we evaluate the inverse ccdf~\(H(z)\) as
\[H(z) = \left(\log{z}\right)^{1/\lambda}\] and since~\(\log{N} = \log(nr_n^2)\) is of the order of~\(\log{n},\) we also deduce that the term~\(\alpha_n = H(N^2)\) and~\(w_n = H(N^{1-\varepsilon})\) satisfy~\[D_1 h_n \leq w_n \leq \alpha_n \leq D_2 h_n\] for some constants~\(D_1,D_2 > 0,\) where~\(h_n =(\log{n})^{1/\lambda}\) is as in Corollary statement. Since the edge weights have bounded~\(s^{th}\) moment for any~\(s > 0,\) given~\(b > 2,\) we choose~\(s > 0\) large enough so that~\(b > \frac{2(s+1)}{s-2}.\) The bound~(\ref{dev_complete}) then implies~(\ref{example_two}) and this completes the proof of the Corollary.~\(\qed\)

\subsection*{\em Data Availability Statement}
Data sharing not applicable to this article as no datasets were generated or analysed during the current study.

\subsection*{\em Acknowledgement}
I thank Professors Rahul Roy, Thomas Mountford and Federico Camia for crucial comments that led to an improvement of the paper. I also thank IISER Bhopal, IMSc and University of Bristol for my fellowships.

\subsection*{\em Conflict of Interest and Funding Statement}
I certify that there is no actual or potential conflict of interest in relation to this article. No funding was received in the preparation of this manuscript.

\bibliographystyle{plain}

\begin{thebibliography}{10}



\bibitem{alon} N. Alon and J. Spencer. (2008).
\newblock{\em The Probabilistic Method}.
\newblock{Wiley Interscience}.

\bibitem{castro} Q. Duchemin and Y. De Castro. (2023).
\newblock{Random Geometric Graph: Some Recent Developments and Perspectives}.
\newblock{\em In: R. Adamczak, N. Gozlan, K. Lounici, M. Madiman (eds) High Dimensional Probability IX. Progress in Probability}, \textbf{80}, Birkh\"auser, 347--392.


\bibitem{kesten} J. T. Cox, A. Gandolfi, P. S. Griffin and  H. Kesten. (1993).
\newblock{Greedy Lattice Animals~\(I:\) Upper Bounds}.
\newblock{\em Annals of Applied Probability}, \textbf{3}, 1151--1169.

\bibitem{ganesan} G. Ganesan. (2025).
\newblock{Deviation Estimates for Extremal Relay Random Geometric Graphs}.
\newblock{\em Indian Journal of Pure and Applied Mathematics}, https://doi.org/10.1007/s13226-025-00780-y.

\bibitem{ganesan2} G. Ganesan. (2025).
\newblock{Minimum Spanning Trees of Random Geometric Graphs with Independent Edge Weights}.
\newblock{\em Sankhya A}, https://doi.org/10.1007/s13171-025-00423-8.

\bibitem{goldsmith} A. Goldsmith. (2005).
\newblock{\em Wireless Communications}.
\newblock{Cambridge University Press}.

\bibitem{grimmett} G. Grimmett. (1989).
\newblock{\em Percolation}.
\newblock{Springer}.


\bibitem{gupta} P. Gupta and P. R. Kumar. (1998).
\newblock {Critical Power for Asymptotic Connectivity in Wireless Networks}.
\newblock {\em Stochastic Analysis, Control, Optimization and Applications}, pp. 2203--2214.

\bibitem{hirsch} C. Hirsch, D. Neuhauser, C. Gloaguen and V. Schmidt. (2015).
\newblock{First-passage Percolation on Random Geometric Graphs and an Application to Shortest-path Trees}.
\newblock{\em Advances in Applied Probability}, \textbf{47}, 328--354.

\bibitem{lima} L. R. de Lima and D. Valesin. (2025).
\newblock{Speed of Convergence and Moderate Deviations of FPP on Random Geometric Graphs}.
\newblock{\em Stochastic Processes and their Applications}, \textbf{187}, 104652.


\bibitem{penrose} M. Penrose. (2003).
\newblock{\em Random Geometric Graphs}.
\newblock{Oxford University Press}.


\bibitem{penrose2} M. Penrose. (2016).
\newblock{Connectivity of Soft Random Geometric Graphs}.
\newblock{\em Annals of Applied Probability}, \textbf{26}, pp. 986--1028.








\end{thebibliography}

\end{document}